\documentclass[a4paper, 11pt]{article}

\usepackage[utf8]{inputenc}
\usepackage{graphicx}
\usepackage{amssymb}
\usepackage{amsmath}
\usepackage{mathtools}
\usepackage[usenames,dvipsnames,svgnames,table]{xcolor}
\usepackage{enumerate}
\usepackage{enumitem}
\setlist[enumerate, 1]{label=(\roman*)}
\setlist[itemize]{leftmargin=1.5em}
\setlist[description]{leftmargin=1em}
\setlist[itemize, 1]{label=$\blacktriangleright$}
\setlist[itemize, 2]{label=$\bullet$}
\usepackage[normalem]{ulem} 

\usepackage{ifthen}

\usepackage{graphicx}
\usepackage{float}
\graphicspath{ {pics/} }

\usepackage[hidelinks]{hyperref}  

\usepackage{amsthm}

\theoremstyle{plain}
\newtheorem{theorem}{Theorem}[section]

\newtheorem{proposition}[theorem]{Proposition}

\theoremstyle{definition}
\newtheorem{definition}{Definition}[section]

\theoremstyle{remark}
\newtheorem{remark}{Remark}[section]

\newcommand{\Z}{\mathbb{Z}}

\newcommand{\R}{\mathbb{R}}

\newcommand{\setbar}{\ | \ }
\newcommand{\set}[2]{\left\{#1 \setbar #2 \right\}}

\newcommand{\scalarprod}[2]{\langle #1, #2 \rangle}

\newcommand{\multiratio}[8]{\frac{|#1 #2|}{|#2 #3|} \cdot \frac{|#3 #4|}{|#4 #5|} \cdot \frac{|#5 #6|}{|#6 #7|} \cdot \frac{|#7 #8|}{|#8 #1|}}

\renewcommand{\refeq}[1]{(\ref{#1})}

\newcommand{\bela}[1]{\begin{equation}\label{#1}}
\newcommand{\ela}{\end{equation}}
\newcommand{\bear}[1]{\begin{array}{#1}}
\newcommand{\ear}{\end{array}}
\newcommand{\bn}{\mbox{\boldmath $n$}}
\newcommand{\bu}{\mbox{\boldmath $u$}}
\newcommand{\be}{\mbox{\boldmath $e$}}
\newcommand{\bef}{\mbox{\boldmath $f$}}
\newcommand{\n}{\mbox{\boldmath $n$}}

\newcommand{\q}{\mbox{\boldmath $x$}}
\newcommand{\sqr}[1]{{(#1)}_{1/2}}
\newcommand{\as}{\\[.6em]}

\newcommand{\AAS}{\\[1.8em]}
\newcommand{\dis}{\displaystyle}
\newcommand{\hn}[1]{\mbox{\boldmath $n$}^{\mbox{\tiny$#1$}}}
\newcommand{\bsigma}{\boldsymbol{\sigma}}

\newcommand{\ignore}[1]{}

\title{On a discretization of confocal quadrics.\linebreak I.\ An integrable systems approach}
\author{Alexander I. Bobenko$^1$, Wolfgang K. Schief$^2$,\\ Yuri B. Suris$^1$, Jan Techter$^1$ \bigskip\\  
$^1$Institut f\"ur Mathematik, TU Berlin, \\ Str.\@ des 17.\@ Juni 136, 10623 Berlin, Germany\bigskip\\
$^2$School of Mathematics and Statistics and\\ Australian Research Council Centre of Excellence for\\ Mathematics and Statistics of Complex Systems,\\ The University of New South Wales, Sydney, NSW 2052, Australia}
\date{\today}

\begin{document}

\maketitle

\begin{abstract}
Confocal quadrics lie at the heart of the system of confocal coordinates (also called elliptic coordinates, after Jacobi).
We suggest a discretization which respects two crucial properties of confocal coordinates: separability and all two-dimensional coordinate subnets being isothermic surfaces (that is, allowing a conformal parametrization along curvature lines, or, equivalently, supporting orthogonal Koenigs nets). Our construction is based on an integrable discretization of the Euler-Poisson-Darboux equation and leads to discrete nets with the separability property, with all two-dimensional subnets being Koenigs nets, and with an additional novel discrete analog of the orthogonality property. The coordinate functions of our discrete nets are given explicitly in terms of gamma functions.
\end{abstract}

\newpage

\tableofcontents

\newpage

\begin{figure}[H]
  \centering
  \includegraphics[clip, trim={380 0 380 0}, width=0.49\textwidth]{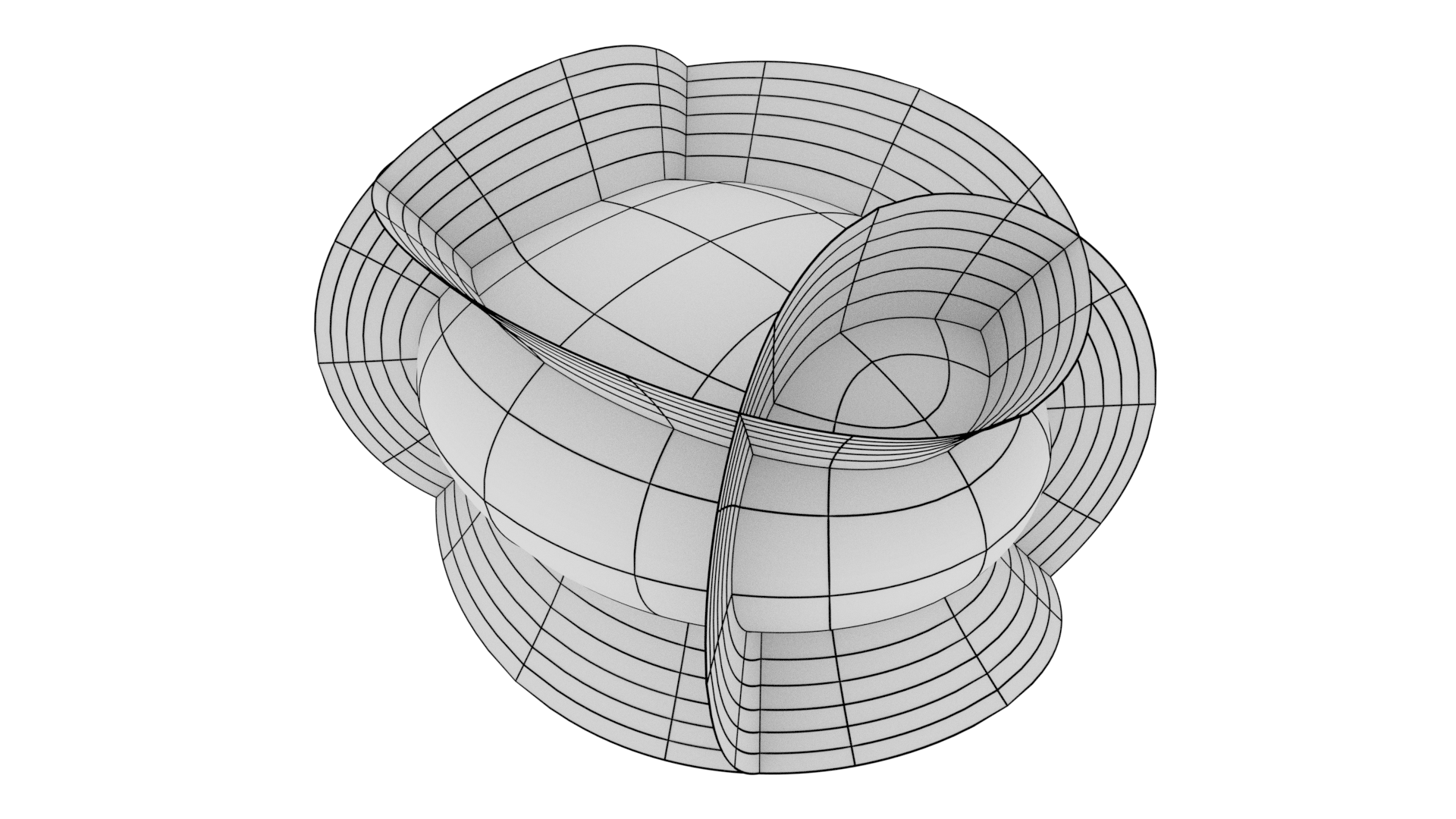}
  \includegraphics[clip, trim={380 0 380 0}, width=0.49\textwidth]{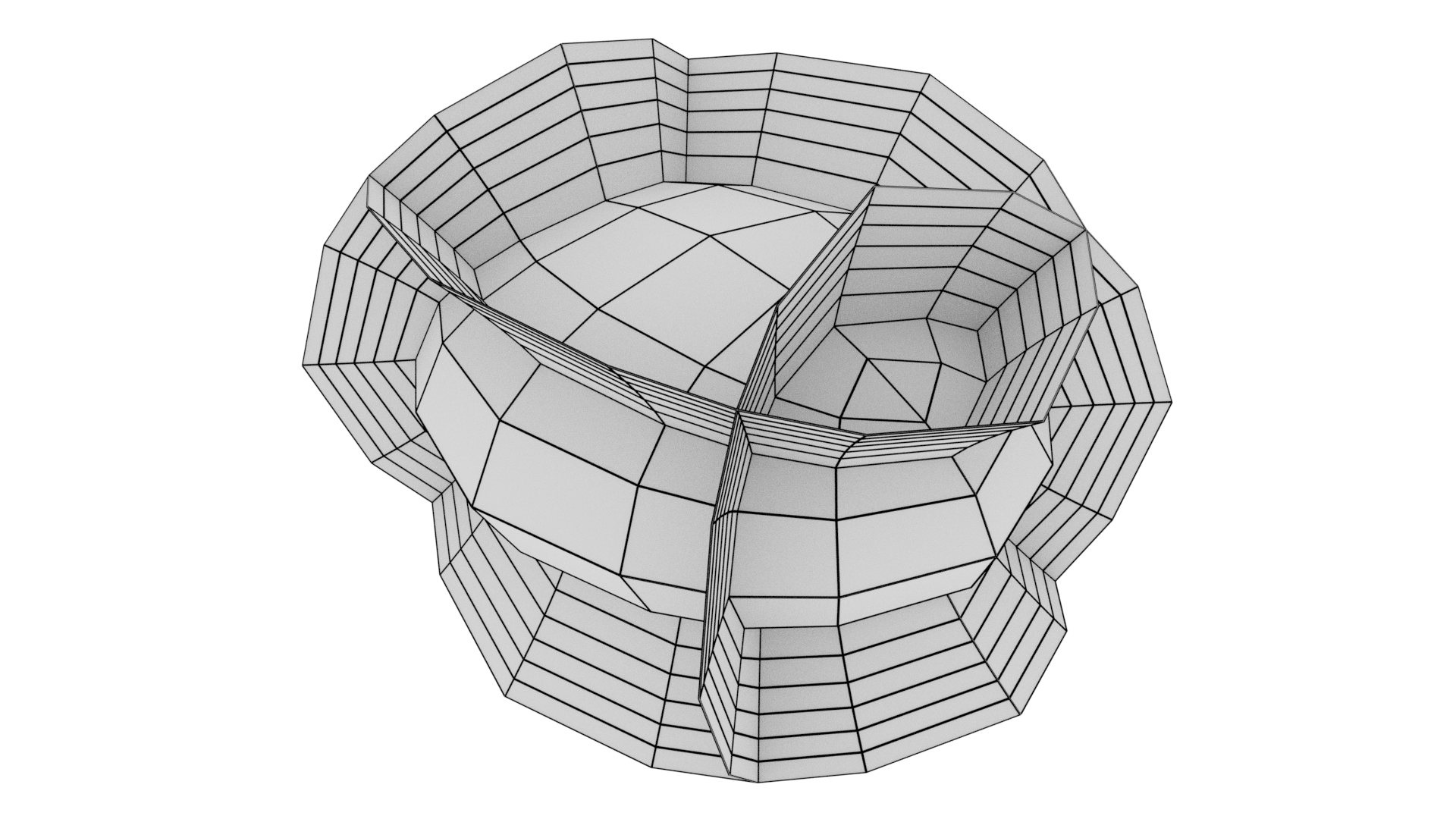}
  \caption{
    Left: three quadrics of different signature from a family of confocal quadrics in $\R^3$.
    Right: the corresponding three discrete quadrics from a family introduced in the present paper.
  }
  \label{fig:confocal3d}
\end{figure}

\section{Introduction}

Confocal quadrics in $\R^N$ belong to the favorite objects of classical mathematics, due to their beautiful geometric properties and numerous relations and applications to various branches of mathematics. To mention just a few well-known examples:
\begin{itemize}
\item Optical properties of quadrics and their confocal families were discovered by the ancient Greeks and continued to fascinate mathematicians for many centuries, culminating in the famous Ivory and Chasles theorems from 19th century given a modern interpretation by Arnold \cite{Arnold}.
\item Dynamical systems: integrability of geodesic flows on quadrics (discovered by Jacobi) and of billiards in quadrics was given a far reaching generalization, with applications to the spectral theory, by Moser \cite{Moser}.
\item Gravitational properties of ellipsoids were studied in detail starting with Newton and Ivory, see \cite[Appendix 15]{Arnold}, \cite[Part 8]{Fuchs-Tabachnikov}, and are based to a large extent on the geometric properties of confocal quadrics.
\item Quadrics in general and confocal systems of quadrics in particular serve as favorite objects in differential geometry. They deliver a non-trivial example of isothermic surfaces which form one of the most interesting classes of ``integrable'' surfaces, that is, surfaces described by integrable differential equations and possessing a rich theory of transformations with remarkable permutability properties.
\item Confocal quadrics lie at the heart of the system of confocal coordinates which allows for separation of variables in the Laplace operator. As such, they support a rich theory of special functions including Lam\'e functions and their generalizations \cite{Bateman-Erdelyi}.  
\end{itemize}

In the present paper, we are interested in a discretization of a system of confocal quadrics, or, what is the same, of a system of confocal coordinates in $\R^n$. According to the philosophy of structure preserving discretization \cite{bobenko-suris}, it is crucial not to follow the path of a straightforward discretization of differential equations, but rather to discretize a well chosen collection of essential geometric properties. 
In the case of confocal quadrics, the choice of properties to be preserved in the course of discretization becomes especially difficult, due to the above-mentioned abundance of complementary geometric and analytic features. 

A number of attempts to discretize quadrics in general and confocal systems of quadrics in particular are available in the literature. In \cite{tsukerman} a discretization of the defining property of a conic as an image of a circle under a projective transformation is considered. Since a natural discretization of a circle is a regular polygon, one ends up with a class of discrete curves which are projective images of regular polygons. 
More sophisticated geometric constructions are developed in \cite{Akopyan-Bobenko} and lead to a very interesting class of quadrilateral nets in a plane and in space, with all quadrilaterals possessing an incircle, resp. all hexahedra possessing an inscribed sphere.  The rich geometric content of these constructions still waits for an adequate analytic description. 

Our approach here is based on a discretization of the classical Euler-Poisson-Darboux equation which has been introduced in \cite{konopelchenko-schief} in the context of discretization of semi-Hamiltonian systems of hydrodynamic type. The discrete Euler-Poisson-Darboux equation is integrable in the sense of multi-dimensional consistency \cite{bobenko-suris}, which, in turn, gives rise to Darboux-type transformations with remarkable permutability properties. 
As we will demonstrate, the integrable nature of the discrete Euler-Poisson-Darboux equation is reflected in the preservation of a suite of algebraic and geometric properties of the confocal coordinate systems.

Our proposal takes as a departure point two properties of the confocal coordinates: they are separable, and all two-dimensional coordinate subnets are isothermic surfaces (which is equivalent to being conjugate nets with equal Laplace invariants and with orthogonal coordinate curves). We propose here a novel concept of discrete isothermic nets. Remarkably, the incircular nets of \cite{Akopyan-Bobenko} turn out to be another instance of this geometry, see Appendix \ref{sect: incircular}. Discretization of confocal coordinate systems based on more general curvature line parametrizations will be addressed in \cite{BSST-II}.

\medskip
{\bf Acknowledgements.} This research is supported by the DFG Collaborative Research Center TRR 109 ``Discretization in Geometry and Dynamics''. We would like to acknowledge the stimulating role of the research visit to TU Berlin by I. Taimanov in summer 2014. In our construction, we combine the
fact that the confocal coordinate system satisfies continuous Euler-Poisson-Darboux equations, which we learned from I. Taimanov, with the  discretization of Euler-Poisson-Darboux equations recently proposed (in a different context) by one of the authors \cite{konopelchenko-schief}.

The pictures in this paper were generated using {\tt blender}, {\tt matplotlib}, {\tt geogebra}, and {\tt inkscape}.

\section{Euler-Poisson-Darboux equation}

\begin{definition}
  Let $U \subset \R^M$ be open and connected.
  We say that a net
  \begin{equation*}
    \q : U \rightarrow \R^N, \quad (u_1, \ldots, u_M) \mapsto (x_1, \ldots, x_N)
  \end{equation*}
  satisfies the \emph{Euler-Poisson-Darboux system}
  if all its two-dimensional subnets satisfy the (vector) Euler-Poisson-Darboux equation with the same parameter $\gamma$:
  \begin{equation}
    \frac{\partial^2 \q}{\partial u_i \partial u_j}
    = \frac{\gamma}{u_i - u_j} \left( \frac{\partial \q}{\partial u_j} - \frac{\partial \q}{\partial u_i} \right)
    \tag{EPD$_\gamma$}
    \label{eq:Euler-Darboux}
  \end{equation}
  for all $i,j \in\{ 1, \ldots, M\}$, $i \neq j$.
\end{definition}

For any $s$ distinct indices $i_1, \ldots, i_s \in \{ 1, \ldots, M\}$, we write
\begin{equation*}
  U_{i_1 \ldots i_s} \coloneqq \set{ (u_{i_1}, \ldots, u_{i_s}) \in \R^s }{ (u_1, \ldots, u_M) \in U }.
\end{equation*}

\begin{definition}
  A two-dimensional subnet of a net $\q : \R^M \supset U \rightarrow \R^N$ corresponding to the coordinate directions
  $i,j\in\{ 1, \ldots, M\}$, $i \neq j$, is called \emph{a Koenigs net}, or, classically, \emph{a conjugate net with equal Laplace invariants},
  if there exists a function
 $
    \nu: U_{ij} \rightarrow \R_+
 $
  such that
  \begin{equation}
    \frac{\partial^2 \q}{\partial u_i \partial u_j}
    = \frac{1}{\nu} \frac{\partial \nu}{\partial u_j} \frac{\partial \q}{\partial u_i}
    + \frac{1}{\nu} \frac{\partial \nu}{\partial u_i} \frac{\partial \q}{\partial u_j}.
    \label{eq:Koenigs-net}
  \end{equation}
  \end{definition}

\begin{proposition}
  \label{prop:Koenigs}
  Let $\q : \R^M \supset U \rightarrow \R^N$ be a net satisfying the Euler-Poisson-Darboux system {\rm \refeq{eq:Euler-Darboux}}.
  Then all two-dimensional subnets of $\q$ are Koenigs nets.
\end{proposition}

\begin{proof}
  The function
   $
    \nu(u_i, u_j) = |u_i - u_j|^\gamma
 $
  solves 
   \begin{equation}
    \label{eq:Koenigs-condition}
      \frac{1}{\nu} \frac{\partial \nu}{\partial u_i} = \frac{\gamma}{u_i - u_j},\qquad
      \frac{1}{\nu} \frac{\partial \nu}{\partial u_j} = \frac{\gamma}{u_j - u_i},
  \end{equation}
 thus the Euler-Poisson-Darboux system \refeq{eq:Euler-Darboux} is of the Koenigs form  \refeq{eq:Koenigs-net}.
\end{proof}

\begin{remark}
  Eisenhart  classified conjugate nets in $\R^3$ with all two-dimensional coordinate surfaces being Koenigs nets \cite{eisenhart}.
  The generic case is described by solutions of the Euler-Poisson-Darboux system \refeq{eq:Euler-Darboux} with an arbitrary coefficient $\gamma$.
\end{remark}

\section{Confocal coordinates}

For given $a_1 > a_2 > \cdots > a_N > 0$,
we consider the one-parameter family of confocal quadrics in $\R^N$ given by
\begin{equation}   \label{eq:confocal-family}
 Q_\lambda=\left\{ \q=(x_1,\ldots,x_N)\in\mathbb R^N:\ \sum_{k=1}^N \frac{x_k^2}{a_k + \lambda} = 1\right\}, \quad \lambda \in \R.
\end{equation}
Note that the quadrics of this family are centered at the origin and have the principal axes aligned along the coordinate directions.
For a given point $\q=(x_1, \ldots, x_N) \in \R^N$ with $x_1x_2\ldots x_N\neq 0$, equation $\sum_{k=1}^N x_k^2/(a_k + \lambda) = 1$ is, after clearing the denominators, a polynomial equation of degree $N$ in $\lambda$, with $N$ real roots $u_1,\ldots, u_N$ lying in the intervals
\begin{equation*}
  -a_1 < u_1 < -a_2 < u_2 < \cdots < -a_N < u_N,
\end{equation*}
so that
\begin{equation} \label{eq: equation factorized}
  \sum_{k=1}^N \frac{x_k^2}{\lambda+a_k}  - 1= - \frac{\prod_{m=1}^N (\lambda - u_m)}{\prod_{m=1}^N (\lambda+a_m) }.
\end{equation}
These $N$ roots correspond to the $N$ confocal quadrics of the family \eqref{eq:confocal-family} that intersect at the point $\q=(x_1, \ldots, x_N)$:
\begin{equation}   \label{eq:confocal-quadrics}
  \sum_{k=1}^N \frac{x_k^2}{a_k + u_i} = 1, \quad i=1,\ldots,N \quad\Leftrightarrow\quad \q\in \bigcap_{i=1}^N Q_{u_i}.
\end{equation}
Each of the quadrics $Q_{u_i}$ is of a different signature. Evaluating the residue of the right-hand side of \eqref{eq: equation factorized} at $\lambda=-a_k$, one can easily express $x_k^2$ through $u_1, \ldots, u_N$:
\begin{equation} \label{eq:elliptic-coordinates-squares}
  x_k^2 = \frac{\prod_{i=1}^N (u_i+a_k)}{\prod_{i \neq k} (a_k - a_i)}, \quad k = 1, \ldots, N.
 \end{equation}
Thus, for each point $(x_1, \ldots, x_N) \in \R^N$ with $x_1x_2\ldots x_N\neq 0$, there is exactly one solution $(u_1, \ldots, u_N) \in \mathcal{U}$ of \refeq{eq:elliptic-coordinates-squares},
where
\begin{equation}
  \mathcal{U} = \set{ (u_1,\ldots,u_N) \in \R^N }{ -a_1 < u_1 < -a_2 < u_2 < \ldots < -a_N < u_N }.
  \label{eq:domain}
\end{equation}
On the other hand, for each $(u_1, \ldots, u_N) \in \mathcal{U}$ there are exactly $2^N$ solutions $(x_1, \ldots, x_N) \in \R^N$,
which are mirror symmetric with respect to the coordinate hyperplanes. In what follows, when we refer to a solution of \refeq{eq:elliptic-coordinates-squares}, we always mean the solution with values in
\begin{equation*}
  \R_+^N = \set{ (x_1, \ldots, x_N) \in \R^N }{ x_1 > 0, \ldots, x_N > 0 }.
\end{equation*}
Thus, we are dealing with a parametrization of the first hyperoctant of $\R^N$,
$\q : \mathcal{U}\ni (u_1, \ldots, u_N) \mapsto (x_1,\ldots,x_N)\in\R_+^N$, given by
\begin{equation} \label{eq:elliptic-coordinates}
  x_k = \frac{\prod_{i=1}^{k-1} \sqrt{-(u_i+a_k)}\prod_{i=k}^N\sqrt{u_i+a_k}}{\prod_{i = 1}^{k-1} \sqrt{a_i - a_k}\prod_{i=k+1}^N\sqrt{a_k-a_i}}, 
  \quad k = 1, \ldots, N,
 \end{equation}
such that the coordinate hyperplanes $u_i = {\rm const}$ are mapped to the respective quadrics given by \refeq{eq:confocal-quadrics}.
The coordinates $(u_1,\ldots,u_N)$ are called \emph{confocal coordinates} (or {\em elliptic coordinates}, following Jacobi \cite[Vorlesung 26]{jacobi}).

\subsection{Confocal coordinates and isothermic surfaces}

\begin{proposition} \label{prop:confocal-Euler-Darboux}
  The net $\q : \mathcal{U} \rightarrow \R_+^N$ given by \refeq{eq:elliptic-coordinates} satisfies the Euler-Poisson-Darboux system {\rm \refeq{eq:Euler-Darboux}} with $\gamma = \frac{1}{2}$.  As a consequence, all two-dimensional subnets of $\q$ are Koenigs nets.
\end{proposition}

\begin{proof}
 The partial derivatives of \refeq{eq:elliptic-coordinates} satisfy
  \begin{equation}   \label{eq:first-derivatives}
    \frac{\partial x_k}{\partial u_i} = \frac{1}{2} \frac{x_k}{(a_k + u_i)}.
  \end{equation}
  From this we compute the second order partial derivatives for $i \neq j$:
  \begin{align}
    \label{eq:second-derivatives}
      \frac{\partial^2 x_k}{\partial u_i \partial u_j}
      &= \frac{1}{2(a_k + u_i)} \frac{\partial x_k}{\partial u_j}
      \ = \ \frac{x_k}{4(a_k + u_i)(a_k + u_j)} \nonumber \\
      &= \frac{x_k}{4(u_i-u_j)} \left( \frac{1}{a_k+u_j} - \frac{1}{a_k+u_i} \right) \nonumber \\
      &= \frac{1}{2(u_i - u_j)} \left( \frac{\partial x_k}{\partial u_j} - \frac{\partial x_k}{\partial u_i} \right). \qedhere
    \end{align}
 \end{proof}

\begin{proposition}
  \label{prop:orthogonal}
  The net $\q : \mathcal{U} \rightarrow \R_+^N$ given by \refeq{eq:elliptic-coordinates} is orthogonal, 
  and thus gives a curvature line parametrization of any of its two-dimensional coordinate surfaces.
\end{proposition}

\begin{proof}
 With the help of \eqref{eq:first-derivatives}, we compute, for $i \neq j$, the scalar product
  \begin{align*}
    \left\langle \frac{\partial \q}{\partial u_i}, \frac{\partial \q}{\partial u_j} \right\rangle
    &= \frac{1}{4} \sum_{k=1}^{N} \frac{x_k^2}{(a_k + u_i)(a_k + u_j)}\\
    &= \frac{1}{4(u_i - u_j)} \sum_{k=1}^{N} \left( \frac{x_k^2}{a_k + u_j} - \frac{x_k^2}{a_k + u_i} \right) = 0,
  \end{align*}
  since $\q(u_1, \ldots, u_N)$ satisfies \refeq{eq:confocal-quadrics} for $u_i$ and for $u_j$.
\end{proof}

We recall the following classical definition.

\begin{definition} \label{dfn:is}
A curvature line parametrized surface $\q: U_{ij}\to\mathbb R^N$ is called an {\em isothermic surface} if its first fundamental form is conformal, possibly upon a reparametrization of independent variables $u_i\mapsto\varphi_i(u_i)$, $u_j\mapsto\varphi_j(u_j)$, that is, if 
$$
\frac{|\partial \q /\partial u_i |^2}{|\partial \q/\partial u_ j |^2}=\frac{\alpha_i(u_i)}{\alpha_j(u_j)}
$$ 
at every point $(u_i,u_j)\in U_{ij}$.
\end{definition}
In other words, isothermic surfaces are characterized by the relations $\partial^2 \q/\partial u_i\partial u_j\!\in{\rm span}(\partial \q/\partial u_i,\partial \q/\partial u_j)$ and
\begin{equation}\label{eq:is prop}
\left\langle\frac{\partial \q}{\partial u_i},\frac{\partial \q}{\partial u_j}\right\rangle=0,\quad 
\left|\frac{\partial \q}{\partial u_i}\right|^2=\alpha_i(u_i) s^2,\quad 
\left|\frac{\partial \q}{\partial u_j}\right|^2=\alpha_j(u_j) s^2,
\end{equation}
with a conformal metric coefficient $s:U_{ij}\to\mathbb R_+$ and with the functions $\alpha_i$, $\alpha_j$, each depending on the respective variable $u_i$, $u_j$ only. These conditions may be equivalently represented as
\begin{equation}\label{eq:is prop1}
\frac{\partial^2 \q}{\partial u_i\partial u_j}=
\frac{1}{s}\frac{\partial s}{\partial u_j} \frac{\partial \q}{\partial u_i}+
\frac{1}{s}\frac{\partial s}{\partial u_i} \frac{\partial \q}{\partial u_j}, \quad
\left\langle\frac{\partial \q}{\partial u_i}, \frac{\partial \q}{\partial u_j}\right\rangle=0.
\end{equation}
Comparison with \eqref{eq:Koenigs-net} shows that isothermic surfaces are nothing but orthogonal Koenigs nets.

Thus, Propositions \ref{prop:confocal-Euler-Darboux}, \ref{prop:orthogonal} immediately imply the first statement of the following proposition.

\begin{proposition} \label{prop: isothermic}
All two-dimensional coordinate surfaces $\q:\mathcal U_{ij}\to\mathbb R^N$ (for fixed values of $u_m$, $m\neq i,j$) of a confocal coordinate system are isothermic. Specifically, one has \eqref{eq:is prop} with 
\begin{equation}  \label{eq: s}
  s=s(u_i,u_j)=|u_i-u_j |^{1/2},
\end{equation}
\begin{equation} \label{eq: alphas}
\frac{\alpha_i(u_i)}{\alpha_j(u_j)} = -\frac{\prod_{m\neq i,j}(u_i - u_m)}{\prod_{m=1}^N (u_i + a_m)}\cdot
 \frac{\prod_{m=1}^N (u_j + a_m)}{\prod_{m\neq i,j}(u_j - u_m)}.
\end{equation}
\end{proposition} 
\begin{proof}
Differentiate both sides of \eqref{eq: equation factorized} with respect to $u_i$. Taking into account \eqref{eq:first-derivatives}, we find:
\begin{equation*}
  \sum_{k=1}^N \frac{x_k^2}{(u_i+a_k)(\lambda+a_k)}
  =  \frac{\prod_{m\neq i} (\lambda - u_m)}{\prod_{m=1}^N (\lambda+a_m) }.
\end{equation*}
Setting $\lambda=u_i$, we finally arrive at 
\begin{equation} \label{eq: id squares}
  \left | \frac{\partial \q}{\partial u_i} \right |^2 =\sum_{k=1}^N\left(\frac{\partial x_k}{\partial u_i}\right)^2
  =\frac{1}{4}\sum_{k=1}^N \frac{x_k^2}{(u_i+a_k)^2}
  =  \frac{1}{4}\frac{\prod_{m\neq i} (u_i - u_m)}{\prod_{m=1}^N (u_i+a_m) }.
\end{equation}
This proves \eqref{eq:is prop} with \eqref{eq: s}, \eqref{eq: alphas}.
%
%
\end{proof}

\begin{remark}  
 
  Darboux classified orthogonal nets in $\R^3$ whose two-dimensional coordinate surfaces are isothermic \cite[Livre II, Chap. III--V]{darboux} . He found several families, all satisfying the Euler-Poisson-Darboux system with coefficient $\gamma = \pm \frac{1}{2}, -1$, or $-2$.  The family corresponding to $\gamma=\frac{1}{2}$ consists of confocal cyclides and includes the case of confocal quadrics (or their M\"obius images).
\end{remark}

\subsection{Confocal coordinates and separability}

From \refeq{eq:elliptic-coordinates} we observe that confocal coordinates are described by very special (separable) solutions of Euler-Poisson-Darboux equations \refeq{eq:Euler-Darboux}. We will now show that the separability property is almost characteristic for confocal coordinates.

\begin{proposition}
  \label{prop:Euler-Darboux-separable-solutions}
  A separable function $x : \R^N \supset U \rightarrow \R$, 
  \begin{equation}
    \label{eq:separable-solution}
    x(u_1, \ldots, u_N) = \prod_{i=1}^N \rho_i(u_i) 
  \end{equation}
  is a solution of the Euler-Poisson-Darboux system {\rm \refeq{eq:Euler-Darboux}} iff the functions $\rho_i : U_i \rightarrow \R$, $i = 1, \ldots, N$, satisfy
  \begin{equation}
    \label{eq:separable-solution-dgl}
    \frac{\rho_i^\prime(u_i)}{\rho_i(u_i)} = \frac{\gamma}{c + u_i}
  \end{equation}
  for some $c \in \R$ and for all $u_i \in U_i$.
\end{proposition}

\begin{proof}
  Computing the derivatives of \refeq{eq:separable-solution} for $i = 1,\ldots,N,$ we obtain:
  \begin{equation*}
    \frac{\partial x}{\partial u_i} = x\cdot \frac{\rho_i^{\prime}(u_i)}{\rho_i(u_i)},
  \end{equation*}
  and for the second derivatives ($i \neq j$):
  \begin{equation}
    \label{eq:separable-second-derivatives1}
    \frac{\partial^2 x}{\partial u_i \partial u_j} 
    = x\cdot \frac{\rho_i^\prime(u_i)}{\rho_i(u_i)} \frac{\rho_j^\prime(u_j)}{\rho_j(u_j)}.
  \end{equation}
  On the other hand, $x$ satisfies the Euler-Poisson-Darboux system \refeq{eq:Euler-Darboux}, which implies
  \begin{equation}
    \label{eq:separable-second-derivatives2}
    \frac{\partial^2 x}{\partial u_i \partial u_j} 
    = \frac{\gamma}{u_i - u_j} \left( \frac{\rho_j^{\prime}(u_j)}{\rho_j(u_j)} - \frac{\rho_i^{\prime}(u_i)}{\rho_i(u_i)} \right) x.
  \end{equation}
  From \refeq{eq:separable-second-derivatives1} and \refeq{eq:separable-second-derivatives2} we obtain
  \[
    u_i - u_j = \gamma \left( \frac{\rho_i(u_i)}{\rho_i^{\prime}(u_i)} - \frac{\rho_j(u_j)}{\rho_j^{\prime}(u_j)} \right),
  \]
  or
  \[  
    \gamma \frac{\rho_i(u_i)}{\rho_i^{\prime}(u_i)} - u_i = \gamma \frac{\rho_j(u_j)}{\rho_j^{\prime}(u_j)} - u_j
  \]
  for all $i, j = 1, ..., N$, $i \neq j$, and $(u_i, u_j) \in U_{ij}$. Thus, both the left-hand side and the right-hand side of the last equation do not depend on $u_i,u_j$.
  So, there exists a $c \in \R$ such that \refeq{eq:separable-solution-dgl} is satisfied.
\end{proof}

For $\gamma = \frac{1}{2}$ general solutions of \refeq{eq:separable-solution-dgl} are given, up to constant factors, by
\begin{alignat*}{2}
  \rho_i(u_i, c) &= \sqrt{u_i +c} &&\quad \text{for} \quad u_i > - c,
  \intertext{respectively by }
  \rho_i(u_i,c) &=\sqrt{-(u_i + c)} &&\quad \text{for} \quad u_i < - c.
\end{alignat*}
A separable solution of the Euler-Poisson-Darboux system \refeq{eq:Euler-Darboux} with $\gamma = \frac{1}{2}$ finally takes the form
\begin{equation*}
  x(u_1, \ldots, u_N) = D \prod_{i=1}^N \rho_i(u_i,c) 
\end{equation*}
with some $c \in \R$, and with a constant $D \in \R$, which is the product of all the constant factors of $\rho_i(u_i,c)$ mentioned above. 

\begin{proposition}
  \label{prop:Euler-Darboux-confocal-quadrics}
  Let $a_1 > \cdots > a_N$ and set 
  \begin{equation*}
    \mathcal{U} = \set{ (u_1,\ldots,u_N) \in \R^N }{ -a_1 < u_1 < -a_2 < u_2 < \cdots < -a_N < u_N }. 
  \end{equation*}
  \begin{itemize}
\item[{\rm a)}]  Let $x_k : \mathcal{U} \rightarrow \R_+$, $k = 1, \ldots, N$, be $N$ independent separable solutions of the Euler-Poisson-Darboux system {\rm \refeq{eq:Euler-Darboux}} with $\gamma = \frac{1}{2}$ defined on $\mathcal U$ and satisfying there the following boundary conditions:
  \begin{align}
    \label{eq:boundary-conditions1}
    \lim_{u_k \searrow \ (- a_k)}x_k(u_1, \ldots, u_N) = 0 \quad &\text{for} \quad k = 1, \ldots, N,\\
    \label{eq:boundary-conditions2}
    \lim_{u_{k-1} \nearrow \ (- a_k)}x_k(u_1, \ldots, u_N) = 0 \quad &\text{for} \quad k = 2, \ldots, N.
  \end{align}
  Then 
  \begin{equation}\label{eq:Euler-Darboux-separable-solutions}
    x_k(u_1, \ldots, u_N) = D_k \prod_{i=1}^N \rho_i(u_i,a_k),\quad k = 1, \ldots, N,
  \end{equation}
  with some $D_1, \ldots, D_N > 0$ and with
  \begin{equation}\label{eq:rhos}
    \rho_i(u_i,a_k)=\left\{\begin{array}{ll}
                             \sqrt{u_i+a_k} \quad & \text{for} \quad i\ge k, \\ \\
                             \sqrt{-(u_i+a_k)} \quad & \text{for} \quad i<k.
                           \end{array}\right.    
  \end{equation}
  Thus, the net $\q =(x_1, \ldots, x_N):\mathcal U\to\R_+^N$ coincides with the confocal coordinates \eqref{eq:elliptic-coordinates} on the positive hyperoctant, 
  up to independent scaling along the coordinate axes $(x_1, \ldots, x_N) \mapsto (C_1 x_1, \ldots, C_N x_N)$ with some $C_1, \ldots, C_N >0$.
  
\item[{\rm b)}] The choice of the constants $D_1,\ldots,D_N > 0$ that specifies the system of confocal coordinates \eqref{eq:elliptic-coordinates} among the separable solutions \eqref{eq:Euler-Darboux-separable-solutions}, namely
\begin{equation}
D_k^{-2}=\prod_{i<k}(a_i-a_k)\prod_{i>k}(a_k-a_i),
\end{equation}
is the unique scaling (up to a common factor) such that the parameter curves are pairwise orthogonal.
\end{itemize}  
\end{proposition}

\begin{proof}
a) We have 
\begin{equation*}
    x_k(u_1, \ldots, u_N) = D_k\cdot \rho_1(u_1,c_k) \cdot \ldots \cdot \rho_N(u_N,c_k),\quad k = 1, \ldots, N.
  \end{equation*}
Boundary conditions \refeq{eq:boundary-conditions1}, \eqref{eq:boundary-conditions2} yield that the constants are given by $c_k = a_k$, and that 
the solutions are actually given by \eqref{eq:Euler-Darboux-separable-solutions}. Formulas \refeq{eq:elliptic-coordinates} are now equivalent to a specific choice of the constants $D_k$.

b)  From \eqref{eq:first-derivatives} we compute:   
 \begin{equation}\label{eq: cont scalar product}
    \left\langle \frac{\partial\q}{\partial u_i}, \frac{\partial\q}{\partial u_j}\right\rangle
    = \frac{1}{4}\sum_{k=1}^N \frac{x_k^2}{(u_i+a_k)(u_j+a_k)} 
    = \frac{1}{4} \sum_{k=1}^N (-1)^{k-1} D_k^2 \prod_{l\neq i,j} (u_l + a_k).
\end{equation}
  We have:
  \[
  \prod_{l\neq i,j} (u_l + a_k)  = \sum_{m=0}^{N-2} p^{(N-2-m)}_{ij}(\bu)a_k^m,
  \]
where $p_{ij}^{(N-2-m)}(\bu)$ is an elementary symmetric polynomial of degree $N-2-m$ in $u_l$, $l\neq i,j$. Thus,
  \begin{align*}
    \left\langle \frac{\partial\q}{\partial u_i},\frac{\partial\q}{\partial u_j} \right\rangle
    &= \frac{1}{4}\sum_{m=0}^{N-2} \left( \sum_{k=1}^N (-1)^{k-1} a_k^m D_k^2 \right) p^{(N-2-m)}_{ij}(\bu).
  \end{align*}
Since the polynomials $p^{(N-2-m)}_{ij}(\bu)$ are independent on $\mathcal{U}$, the latter expression is equal to zero if and only if
  \[
  \sum_{k=1}^N (-1)^{k-1} a_k^m D_k^2 = 0, \qquad m= 0, \ldots, N-2.
  \]
This system of $N-1$ linear homogeneous equations for the $N$ unknowns $D_k^2$ does not depend on $i,j$. Supplying it by the non-homogeneous equation $\sum_{k=1}^N (-1)^{k-1} a_k ^{N-1} D_k^2 = 1$, we find the unique solution of the resulting system as quotients of Vandermonde determinants, or finally $(-1)^{k-1}D_k^2=1/\prod_{i\neq k} (a_k-a_i)$.
 \end{proof}

\begin{remark}
  The boundary conditions ensure that the $2N-1$ faces of the boundary of $\mathcal{U}$ are mapped into coordinate hyperplanes.
  Their images are degenerate quadrics of the confocal family  \refeq{eq:confocal-family}.
\end{remark}

\section{Discrete Koenigs nets}
\label{sect: dKoenigs}

For a function $f$ on $\Z^M$ we define the \emph{difference operator} in the standard way:
\begin{equation*}
  \Delta_i f(\bn) = f(\bn + \be_i) - f(\bn)
\end{equation*}
for all $\bn \in \Z^M$,
where $\be_i$ is the $i$-th coordinate vector of $\Z^M$.

\begin{definition}
  A two-dimensional discrete net $\q :  \Z^M \supset U_{ij} \rightarrow \R^N$ corresponding to the coordinate directions $i,j \in \{1, \ldots, M\}$, 
  $i \neq j$, is called \emph{a discrete Koenigs net} if there exists a function
  $
    \nu : U_{ij} \rightarrow \R_+
  $
  such that
  \begin{equation}
    \label{eq:discrete-Koenigs}
    \Delta_i \Delta_j \q = \frac{\nu_{(j)} \nu_{(ij)} - \nu \nu_{(i)}}{\nu(\nu_{(i)} + \nu_{(j)})} \Delta_i \q + 
                                     \frac{\nu_{(i)} \nu_{(ij)} - \nu \nu_{(j)}}{\nu(\nu_{(i)} + \nu_{(j)})} \Delta_j \q.
  \end{equation}
Here we use index notation to denote shifts of $\nu$:
$$
    \nu_{(i)}(\bn) \coloneqq \nu(\bn + \be_i),\qquad
    \nu_{(ij)}(\bn) \coloneqq \nu(\bn + \be_i + \be_j), \qquad \bn \in \Z^M.
$$
\end{definition}

The geometric meaning of this algebraic definition is as follows. Like in the continuous case, discrete K\"onigs nets constitute a subclass of discrete conjugate nets (Q-nets), in the sense that all two-dimensional subnets have planar faces. See \cite{bobenko-suris} for more information on Q-nets, as well as on geometric properties of discrete Koenigs nets. 
Consider an elementary planar quadrilateral $(\q,\q_{(i)},\q_{(ij)},\q_{(j)})$ of a Q-net governed by the discrete Darboux equation 
  \begin{equation}
    \label{eq:discrete-darboux}
    \Delta_i \Delta_j \q = A \Delta_i \q + B \Delta_j \q.
  \end{equation}
Let $M$ be the intersection point of its diagonals $[\q,\q_{(ij)}]$ and $[\q_{(i)},\q_{(j)}]$. Then one can easily compute that $M$ divides the corresponding diagonals in the following relations:
  \[
 \overrightarrow{ \q_{(i)}M }: \overrightarrow{M\q_{(j)}} = (B+1) :(A+1), \quad \overrightarrow{\q M}: \overrightarrow{M\q_{(ij)}} = 1:(A+B+1).
  \]
A Q-net is called a Koenigs net, if there is a positive function $\nu$ defined at the vertices of the net such that
\[
\overrightarrow{ \q_{(i)}M} : \overrightarrow{ M\q_{(j)}}=\nu_{(i)} : \nu_{(j)},\quad 
\overrightarrow{  \q M}: \overrightarrow{ M\q_{(ij)}}=\nu:\nu_{(ij)}.
\]
One can show \cite{bobenko-suris} that this happens if and only if the intersection points of the diagonals of four adjacent quadrilaterals are coplanar.
The function $\nu$ should satisfy
\bela{Z2}
  (A+1)\nu_{(i)} = (B+1)\nu_{(j)},\quad \nu_{(ij)} = (A+B+1)\nu.
\ela
This is clearly equivalent to
\bela{Z1}
  A= \frac{\nu_{(j)} \nu_{(ij)} - \nu \nu_{(i)}}{\nu(\nu_{(i)} + \nu_{(j)})},\quad 
  B= \frac{\nu_{(i)} \nu_{(ij)} - \nu \nu_{(j)}}{\nu(\nu_{(i)} + \nu_{(j)})}.
\ela
The pair of linear equations \eqref{Z2} is compatible if and only if the following nonlinear equation is satisfied for the coefficients $A,B$ associated with four adjacent quadrilaterals:
\bela{Z3}
  \frac{A_{(ij)}+1}{B_{(ij)}+1} = \frac{(A_{(j)}+B_{(j)}+1)}{(A_{(i)}+B_{(i)}+1)}\frac{(A+1)}{(B+1)}.
\ela
If this relation for the coefficients $A,B$ of the discrete Darboux equation \eqref{eq:discrete-darboux} holds true everywhere, then the linear equations (\ref{Z2}) determine a function $\nu$ uniquely, as soon as initial data are prescribed, consisting, for instance, of the values of $\nu$ at two neighboring vertices. The associated discrete Darboux equation is then of Koenigs type (\ref{eq:discrete-Koenigs}).

\section{Discrete Euler-Poisson-Darboux equation}

\begin{definition}
  Let $U \subset \Z^M$.
  We say that a discrete net
  \begin{equation*}
    \q : U \rightarrow \R^N, \quad (n_1, \ldots, n_M) \mapsto (x_1, \ldots, x_N)
  \end{equation*}
  satisfies the \emph{discrete Euler-Darboux system} if all of its two-dimensional subnets satisfy the (vector) discrete Euler-Poisson-Darboux equation with the same parameter $\gamma$:
  \begin{equation}
    \label{eq:discrete-Euler-Darboux}
    \Delta_i \Delta_j \q = \frac{\gamma}{n_i + \epsilon_i - n_j - \epsilon_j} ( \Delta_j \q - \Delta_i \q )
    \tag{dEPD$_\gamma$}
  \end{equation}
  for all $i, j \in \{1, \ldots, M\}$, $i \neq j$, and some $\gamma \in \R$, $\epsilon_1, \ldots, \epsilon_M \in \R$.
\end{definition}

\begin{remark}
  This discretization of the Euler-Poisson-Darboux system was introduced by Konopelchenko and Schief  \cite{konopelchenko-schief}.
\end{remark}

\begin{proposition}
  Let $\q : \Z^M \supset U \rightarrow \R^N$ be a discrete net satisfying the discrete Euler-Poisson-Darboux system
 {\rm \refeq{eq:discrete-Euler-Darboux}}.
  Then all two-dimensional subnets of $\q$ are discrete Koenigs nets.
\end{proposition}

\begin{proof}
 It is straightforward to verify that the coefficients 
\bela{Z4}
  A = -B = \frac{\gamma}{n_i + \epsilon_i - n_j - \epsilon_j}
\ela
indeed obey the Koenigs condition (\ref{Z3}).
\end{proof}

We now show that for a discrete net satisfying the discrete Euler-Poisson-Darboux equation \eqref{eq:discrete-Euler-Darboux}, the function $\nu$ can be found explicitly. For this aim, use the ansatz
  \begin{equation*}
    \nu(n_i, n_j) = \mu(n_i - n_j), 
  \end{equation*}
  so that $\nu_{(ij)} = \nu(n_i+1,n_j+1)=\nu(n_i,n_j)=\nu$.
  Under this ansatz, equation \refeq{eq:discrete-Koenigs} simplifies to
  \begin{equation*}
    \Delta_i \Delta_j \q = \frac{\nu_{(j)} - \nu_{(i)}}{\nu_{(i)} + \nu_{(j)}} \Delta_i \q + \frac{\nu_{(i)} - \nu_{(j)}}{\nu_{(i)} + \nu_{(j)}} \Delta_j \q.
  \end{equation*}
  Comparing with \refeq{eq:discrete-Euler-Darboux} we obtain
  \begin{align*}
    &\frac{\nu_{(i)} - \nu_{(j)}}{\nu_{(i)} + \nu_{(j)}} = \frac{\gamma}{n_i + \epsilon_i - n_j - \epsilon_j}\\
    \Leftrightarrow~ &\nu_{(i)} \left( 1 - \frac{\gamma}{n_i + \epsilon_i - n_j - \epsilon_j} \right)
    = \nu_{(j)} \left( 1 + \frac{\gamma}{n_i + \epsilon_i - n_j - \epsilon_j} \right)\\
    \Leftrightarrow~ &\nu(n_i+1,  n_j ) = \nu(n_i, n_j + 1) \ \frac{n_i + \epsilon_i - n_j - \epsilon_j + \gamma}{n_i + \epsilon_i - n_j - \epsilon_j - \gamma}.
      \end{align*}
Thus, the function $\mu$ should satisfy      
$$
    \mu(m + 1) = \mu(m - 1) \ \frac{m + \epsilon_i - \epsilon_j + \gamma}{m + \epsilon_i - \epsilon_j - \gamma}.
$$
This equation is easily solved:
  \begin{equation*}
    \mu(m) =\frac{\Gamma \left( \frac{1}{2} ( m + \epsilon_i - \epsilon_j + \gamma + 1 ) \right)}
    {\Gamma \left( \frac{1}{2} ( m + \epsilon_i - \epsilon_j - \gamma + 1 ) \right)}b(m),
  \end{equation*}
  where $\Gamma$ denotes the gamma function and $b$ is any function of period 2. It is recalled that $\Gamma(x+1)=x\Gamma(x)$.

\section{Discrete confocal quadrics}
\label{sect: discrete quadrics}

We have seen in the continuous case (Proposition \ref{prop:Euler-Darboux-confocal-quadrics}) that confocal quadrics are described,
up to a componentwise scaling,  by separable solutions of the Euler-Poisson-Darboux system \refeq{eq:Euler-Darboux} with $\gamma = \frac{1}{2}$.
It is therefore natural to consider separable solutions of the discrete Euler-Poisson-Darboux system.
\subsection{Separability}

\begin{proposition} {\rm \cite{konopelchenko-schief}}
  A separable function $x : \Z^N \supset U \rightarrow \R$,
  \begin{equation}
    \label{eq:discrete-separable-solution}
    x(n_1, \ldots, n_N) = \rho_1(n_1) \cdots \rho_N(n_N),
  \end{equation}
  is a solution of the discrete Euler-Poisson-Darboux system {\rm \refeq{eq:discrete-Euler-Darboux}} iff the functions $\rho_i: U_i \rightarrow \R$, $i = 1, \ldots, N$, satisfy
  \begin{equation}
    \label{eq:discrete-separable-solution-dgl-delta}
    \Delta\rho_i(n_i) = \frac{\gamma \rho_i(n_i)}{n_i + \epsilon_i + c} ,
  \end{equation}  
  or, equivalently,
  \begin{equation}
    \label{eq:discrete-separable-solution-dgl}
    \rho_i(n_i+1) = \rho_i(n_i) \frac{n_i + \epsilon_i + c + \gamma}{n_i + \epsilon_i + c}
  \end{equation}
  for some $c \in \R$ and for all $n_i \in U_i$.
\end{proposition}

\begin{proof}
  Substituting \refeq{eq:discrete-separable-solution} into \refeq{eq:discrete-Euler-Darboux} we obtain
\begin{align*}
 &\big(\rho_i(n_i+1)-\rho_i(n_i)\big)\big(\rho_j(n_j+1) - \rho_j(n_j) \big)\\
 &\quad= \frac{ \gamma}{ n_i + \epsilon_i - n_j - \epsilon_j }  \Big( \rho_i(n_i)\big(\rho_j(n_j+1) -\rho_j(n_j)\big)-\rho_j(n_j) \big( \rho_i(n_i+1)-\rho_i(n_i)\big)\Big),
\end{align*}
    which is equivalent to
\[
   n_i + \epsilon_i - n_j - \epsilon_j
    = \frac{\gamma \rho_i(n_i)}{\rho_i(n_i+1) - \rho_i(n_i)}
    - \frac{\gamma \rho_j(n_j)}{\rho_j(n_j+1) - \rho_j(n_j)}.
\]
So, the expression 
\[
\frac{\gamma \rho_i(n_i)}{\rho_i(n_i+1) - \rho_i(n_i)}-(n_i+\epsilon_i)=c
\]
depends neither on $n_i$ nor on $n_j$, i.e., is a constant. This is equivalent to \eqref{eq:discrete-separable-solution-dgl-delta} and thus to \eqref{eq:discrete-separable-solution-dgl}.
\end{proof}

If the constants $\gamma$, $c$ and $\epsilon_i$ are such that neither $\epsilon_i + c$  nor $\epsilon_i + c + \gamma$ is an integer, then
the general solution of \refeq{eq:discrete-separable-solution-dgl} is given by
\begin{equation*}
  \rho_i(n_i,\epsilon_i+c) 
  = d_i \frac{\Gamma(n_i + \epsilon_i + c + \gamma)}{\Gamma(n_i + \epsilon_i + c)} 
  = \tilde{d}_i \frac{\Gamma(-n_i - \epsilon_i - c + 1)}{\Gamma(-n_i - \epsilon_i - c - \gamma + 1)}
\end{equation*}
for all $n_i \in \Z$ with some constants $d_i, \tilde{d_i} \in \R$.

In what follows, we will use the Pochhammer symbol $(u)_\gamma$ with a not necessarily integer index $\gamma$:
\begin{equation}\label{eq: Pochhammer gamma}
(u)_\gamma=\frac{\Gamma(u+\gamma)}{\Gamma(u)}, \quad u,\gamma>0.
\end{equation}
Because of the asymptotics $(u)_\gamma\sim u^\gamma$ as $u\to+\infty$, which can also be put as
\[
\lim_{\varepsilon\to 0}\varepsilon^\gamma \left(\frac{u}{\varepsilon}\right)_\gamma=u^\gamma,
\]
the function $(u)_\gamma$ has been considered as a discretization of $u^\gamma$ in \cite[p. 47]{Gelfand_et_al}. With this notation, the above formulas take the form 
\begin{equation*}
  \rho_i(n_i,\epsilon_i+c)  = d_i(n_i + \epsilon_i + c)_\gamma = \tilde{d}_i (-n_i - \epsilon_i - c - \gamma + 1)_\gamma.
\end{equation*}

\begin{definition}
The {\em discrete square root function} is defined by
\begin{equation}
(u)_{1/2}=\frac{\Gamma(u+\frac{1}{2})}{\Gamma(u)}.
\end{equation} 
\end{definition}

To achieve boundary conditions similar to \refeq{eq:boundary-conditions1} and \refeq{eq:boundary-conditions2},
we will be more interested in the cases where solutions are only defined on an integer half-axis, and vanish at its boundary point.
For $\gamma = \frac{1}{2}$ this is the case if:
\begin{itemize}
 \item either $\epsilon_i + c\in\Z$, and  then the general solution to the right of $-c -\epsilon_i$ is given by multiples of
    \begin{equation}
      \label{eq:discrete-separable-solution-version1}
      \rho_i(n_i, \epsilon_i + c) = (n_i + \epsilon_i + c)_{1/2}
      \quad \text{for} \quad n_i \geq -c - \epsilon_i,
      \end{equation}
 \item or $\epsilon_i + c +\frac{1}{2}\in\Z$, and then the general solution to the left of $-c -\epsilon + \frac{1}{2}$ is given by multiples of
   \begin{equation}
      \label{eq:discrete-separable-solution-version2}
      \rho_i(n_i, \epsilon_i + c) = \sqr{-n_i - \epsilon_i - c + \tfrac{1}{2}}
      \quad \text{for} \quad n_i \leq -c - \epsilon_i + \frac{1}{2}.
  \end{equation}
\end{itemize}
In terms of discrete square roots, the expressions for the separable solutions of the discrete Euler-Poisson-Darboux system for $\gamma=\frac{1}{2}$ now resemble their classical counterparts. 

\begin{proposition}\label{prop:boundary}
Let $\alpha_1, \ldots, \alpha_N \in \Z$ with $\alpha_1 > \alpha_2 > \cdots > \alpha_N$, and set
 \begin{equation*}
    \mathcal{U} =\set{ (n_1,\ldots,n_N) \in \Z^N }{ -\alpha_1 \leq n_1 \leq -\alpha_2 \leq n_2 \leq \cdots \leq -\alpha_N \leq n_N } .
 \end{equation*}  
Let ${x_k : \mathcal{U} \rightarrow \R_+}$, $k = 1, \ldots, N$, be $N$ independent separable solutions
of the discrete Euler-Poisson-Darboux system {\rm \refeq{eq:discrete-Euler-Darboux}} with $\gamma = \frac{1}{2}$ defined on $\mathcal U$
and satisfying the following boundary conditions:
  \begin{alignat}{2}
    \label{eq:discrete-boundary-conditions1}
    &x_k|_{n_k=-\alpha_k} = 0  &&\quad \text{for} \quad k= 1, \ldots, N,\\
    \label{eq:discrete-boundary-conditions2}
    &x_k |_{n_{k-1} = -\alpha_k} = 0 &&\quad \text{for} \quad k= 2, \ldots, N.
  \end{alignat}
 Then the shifts $\epsilon_i$ of the variables $n_i$ are given by
  \begin{equation}
    \label{eq:epsilons}
    \epsilon_i - \epsilon_j = \frac{j-i}{2} \quad\text{for} \quad i,j = 1, \ldots, N,
  \end{equation}
  and the solutions are expressed by
  \begin{equation}
    \label{eq:discrete-solutions-with-boundary-conditions1}
    x_k(n_1, \ldots, n_N) = D_k  \prod_{i=1}^N \rho_i^{(k)}(n_i),
  \end{equation} 
  for some constants $D_1,\ldots, D_N > 0$ and
  \begin{equation}
    \label{eq:discrete-solutions-with-boundary-conditions2}
    \rho_i^{(k)}(n_i) =
    \left\{ \begin{array}{ll}    
    \sqr{n_i + \alpha_k + \frac{k-i}{2}} \quad & \text{for} \quad i \geq k, \\ \\ 
    \sqr{-n_i - \alpha_k - \frac{k-i}{2} + \frac{1}{2}}  \quad & \text{for} \quad i < k.
     \end{array}\right.
  \end{equation}  
 \end{proposition}

\begin{proof}
Separable solutions of \refeq{eq:discrete-Euler-Darboux} with $\gamma=\frac{1}{2}$ are of the general form \eqref{eq:discrete-solutions-with-boundary-conditions1},  where each $\rho_i^{(k)}(n_i)=\rho_i(n_i,\epsilon_i+c_k)$ is defined by one of the formulas \eqref{eq:discrete-separable-solution-version1}, \eqref{eq:discrete-separable-solution-version2}, and all multiplicative constants are collected in $D_1,\ldots, D_N > 0$. We have to determine suitable constants $\epsilon_i$ and 
$c_k$.
  
The boundary conditions \refeq{eq:discrete-boundary-conditions1} and \refeq{eq:discrete-boundary-conditions2}  imply that 
$x_k$ is defined for $n_k \geq -\alpha_k$, while vanishing for $n_k = -\alpha_k$, and also that $x_k$ is defined for $n_{k-1} \leq -\alpha_k$, while vanishing for $n_{k-1} = -\alpha_k$. This shows that
\[
\alpha_k = c_k + \epsilon_k=c_k + \epsilon_{k-1} - \frac{1}{2}.
\] 
We obtain $\epsilon_k- \epsilon_{k-1}  = -\frac{1}{2}$, and equation \eqref{eq:epsilons} follows. Together with $c_k + \epsilon_k = a_k$ this implies that
  \begin{equation}
   \label{eq: a vs c}
    c_k + \epsilon_i = \alpha_k + \frac{k-i}{2}.
  \end{equation}
It remains to substitute \eqref{eq: a vs c} into \refeq{eq:discrete-separable-solution-version1} and \refeq{eq:discrete-separable-solution-version2}.
\end{proof}

\subsection{Orthogonality}

The remaining scaling freedom (multiplicative constants $D_k$) of the components $x_k$ as given by (\ref{eq:discrete-solutions-with-boundary-conditions1}) is the same as in the continuous case. As we have seen in Proposition \ref{prop:Euler-Darboux-confocal-quadrics}, in the continuous case, one can fix the scaling by imposing the orthogonality condition $(\partial \q/\partial u_i) \perp (\partial \q/\partial u_j)$. In the discrete case, it 
turns out to be possible to introduce a notion of orthogonality, which will allow us to fix the scaling in a similar way. 

\begin{definition}\label{def: d ortho}
Let $\mathcal U\subset \Z^N$, $\mathcal U^*\subset {(\Z^N)}^*$, where ${(\Z^N)}^*=(\Z+\frac{1}{2})^N$. Consider a function
\bela{Z6a}
  \q : \mathcal{U}\cup\mathcal{U}^* \rightarrow \R^N,
\ela
such that both restrictions $\q(\mathcal U)$ and $\q(\mathcal U^*)$ are Q-nets. We say that the discrete net $\q$ is {\em orthogonal} if each edge of $\q(\mathcal U)$ is orthogonal to the dual facet of $\q(\mathcal U^*)$.
\end{definition}

\begin{figure}[H]
  \centering
  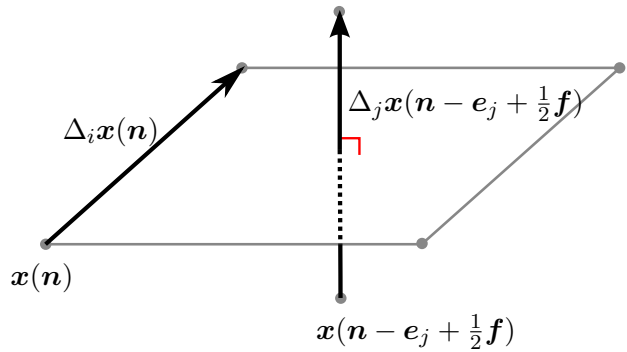
  \caption{Discrete orthogonality for a system of Q-nets defined on a square lattice and on its dual.}
  \label{fig:orthogonality}
\end{figure}

The (space of the) facet of $\q(\mathcal U^*)$ dual to the edge $[\q(\bn),\q(\bn+\be_i)]$ of $\q(\mathcal U)$ is spanned by the $N-1$ edges $[\q(\bn-\be_j+\frac{1}{2}\bef),\q(\bn+\frac{1}{2}\bef)]$ with $j\neq i$, where $\bef=(1,\ldots,1)$. Therefore, the orthogonality condition in the sense of Definition \ref{def: d ortho} reads:
\begin{equation}
    \label{eq:orthoganality-condition}
    \left\langle \Delta_i \q(\bn), \Delta_j \q(\bn-\be_j+\tfrac{1}{2}\bef) \right\rangle = 0
\end{equation}
for all $i \neq j$ and $\bn \in \Z^N$. 
From this it is easy to see that $\q(\mathcal U)$ and $\q(\mathcal U^*)$ actually play symmetric roles in Definition \ref{def: d ortho} (that is, each edge of $\q(\mathcal U^*)$ is orthogonal to the dual facet of $\q(\mathcal U)$, compare Fig.\ \ref{fig:orthogonality}). 
\medskip

Turning to separable solutions of the discrete Euler-Poisson-Darboux system {\rm \refeq{eq:discrete-Euler-Darboux}} with $\gamma = \frac{1}{2}$, we extend the function $\q = (x_1,\ldots,x_N)$ defined in Proposition \ref{prop:boundary} to a bigger domain:
\bela{Z5}
  \q : \mathcal{U}\cup\mathcal{U}^* \rightarrow \R^N_+,
\ela
where
\bela{Z6}
 \mathcal{U}^* =\set{ (n_1,\ldots,n_N) \in {(\Z^N)}^* }{ -\alpha_1 \leq n_1 \leq -\alpha_2 \leq n_2 \leq \cdots \leq -\alpha_N \leq n_N } .
\ela
It is emphasized that the lattices $\q(\mathcal{U})$ and $\q(\mathcal{U}^*)$ are on equal footing except that the boundary conditions do not apply to $\q(\mathcal{U}^*)$.

\begin{proposition}
  \label{prop:discrete-orthogonality}
  Let $\alpha_1, \ldots, \alpha_N \in \Z$ with $\alpha_1 > \alpha_2 > \cdots > \alpha_N$.
  Then the net $\q : \mathcal{U}\cup\mathcal{U}^* \rightarrow \R^N_+$ 
  defined by (\ref{eq:discrete-solutions-with-boundary-conditions1}) 
  and (\ref{eq:discrete-solutions-with-boundary-conditions2}) is orthogonal
 if and only if
  \begin{equation}
    \label{eq:D-solutions}
    D_k^{-2} = C \prod_{i < k} (\alpha_i - \alpha_k + \tfrac{i-k}{2}) \prod_{i > k} (\alpha_k - \alpha_i + \tfrac{k-i}{2})
  \end{equation}
  with some $C \in \R_+$.
\end{proposition}

\begin{proof} We will use the following formulas for the ``discrete derivative'' of the ``discrete square root function'' $(u)_{1/2}=\Gamma(u+\tfrac{1}{2})/\Gamma(u)$, which are immediate consequences of the identity $\Gamma(u+1) = u\Gamma(u)$:
\bela{E3}
  \Delta \big(\sqr{u}\big) = \frac{1}{2\sqr{u+\frac{1}{2}}}, \qquad \Delta\big(\sqr{-u}\big) = -\frac{1}{2\sqr{-u-\frac{1}{2}}},
\ela
where $\Delta f(u) = f(u+1) - f(u)$. We also note that the ``discrete squares'' of discrete square roots obey the relations
\bela{E4a}
 \textstyle  \sqr{u}\sqr{u+\frac{1}{2}} = u,\qquad \sqr{-u}\sqr{-u-\frac{1}{2}} = -u-\frac{1}{2}.
\ela

Substituting (\ref{eq: a vs c}) and $\gamma = \frac{1}{2}$ into (\ref{eq:discrete-separable-solution-dgl-delta}), we obtain:
\begin{equation}\label{eq: Delta rho} 
   \Delta \rho^{(k)}_i(n_i) = \frac{\rho^{(k)}_i(n_i)}{2(n_i + \alpha_k + \frac{k - i}{2})}.
\end{equation}
Upon using property \eqref{E4a} and expressions \eqref{eq:discrete-solutions-with-boundary-conditions2}, we arrive at
\begin{equation}\label{eq: rho rho}
\rho_i^{(k)}(n_i)\rho_i^{(k)}(n_i+\tfrac{1}{2})=\left\{\begin{array}{ll} n_i+\alpha_k+\tfrac{k-i}{2}, & i\ge k, \\ \\
-(n_i+\alpha_k+\tfrac{k-i}{2}),  & i<k. \end{array}\right.
\end{equation}
We use \eqref{eq:discrete-solutions-with-boundary-conditions1}, \eqref{eq: Delta rho},  \eqref{eq: rho rho} to compute the left-hand side of equation \eqref{eq:orthoganality-condition}:
\[
   \left\langle  \Delta_i \q(\bn) , \Delta_j \q(\bn-\be_j+\tfrac{1}{2}\bef)\right\rangle 
    = \frac{1}{4} \sum_{k=1}^N (-1)^{k-1} D_k^2 \prod_{l\neq i,j} \left( n_l + \alpha_k + \frac{k-l}{2} \right).
\]
Observe that this literally coincides with the analogous expression in the continuous case \eqref{eq: cont scalar product}, if we set
\begin{equation}
a_k=\alpha_k+\frac{k}{2}, \quad u_l=n_l-\frac{l}{2}.
\end{equation}
Therefore, the proof of part b) of Proposition \ref{prop:Euler-Darboux-confocal-quadrics} can be literally repeated, leading to the condition $D_k^{-2}=C\prod_{i<k} (a_i-a_k)\prod_{i>k} (a_k-a_i)$, which coincides with \eqref{eq:D-solutions}. 
\end{proof}

\subsection{Definition of discrete confocal coordinates}

\begin{definition}
  \label{def:discrete-elliptic-coordinates}
  Let $\alpha_1, \ldots, \alpha_N \in \Z$ with $\alpha_1 > \alpha_2 > \cdots > \alpha_N$. \emph{Discrete confocal coordinates} are given by the discrete net $\q : \mathcal{U}\cup\mathcal{U}^* \rightarrow \R^N_+$ defined by
  \begin{equation}
    \label{eq:discrete-elliptic-coordinates}
    x_k(\bn) = D_k \prod_{i=1}^{k-1} \sqr{-n_i - \alpha_k - \tfrac{k-i}{2} + \tfrac{1}{2}} 
    \prod_{i=k}^N \sqr{n_i + \alpha_k + \tfrac{k-i}{2}} 
  \end{equation}
  with
  \begin{equation}
    \label{eq:D}
    D_k^{-1} = \prod_{i=1}^{k-1} \sqrt{\alpha_i - \alpha_k + \tfrac{i-k}{2}} 
    \prod_{i=k+1}^{N} \sqrt{\alpha_k - \alpha_i + \tfrac{k-i}{2}} .
  \end{equation}
\end{definition}

The characteristic properties of this net can be summarized as follows.
  \begin{itemize} 
  \item[(i)] Each two-dimensional subnet of $\q(\mathcal U)$ and of $\q(\mathcal U^*)$ satisfies  \eqref{eq:discrete-Euler-Darboux} with $\gamma=\frac{1}{2}$;
  \item[(ii)] Therefore each two-dimensional subnet of $\q(\mathcal U)$ and of $\q(\mathcal U^*)$ is a Koenigs net;
  \item[(iii)] The net $\q(\mathcal U\cup\mathcal U^*)$ is orthogonal in the sense of Definition \ref{def: d ortho};
  \item[(iv)] Both nets $\q(\mathcal U)$ and  $\q(\mathcal U^*)$ are separable;
  \item[(v)] Boundary conditions  \eqref{eq:discrete-boundary-conditions1}, \eqref{eq:discrete-boundary-conditions2} are satisfied.
  \end{itemize}
  
Properties (ii) and (iii) lead to a novel discretization of the notion of isothermic surfaces. 
 
Property (v) allows us to define  \emph{discrete confocal quadrics} by reflecting the net $\q$  in the coordinate hyperplanes. Like in the continuous case, the boundary conditions  correspond to the $2N - 1$ degenerate quadrics of the confocal family lying in the coordinate hyperplanes.

\begin{remark}
  In \cite{BSST-II} we will describe more general  discrete confocal quadrics, corresponding to general reparametrizations of the curvature lines. They will be defined in a more geometric way, less based on integrable difference equations.
\end{remark}

\subsection{Further properties of discrete confocal coordinates}
\label{sec:properties}

We now obtain a variety of properties of discrete confocal quadrics and discrete confocal coordinates, which serve as discretizations of their respective continuous analogs.

\begin{proposition}
For any $N$-tuple of signs $\bsigma=(\sigma_1,\ldots,\sigma_N)$, $\sigma_i=\pm 1$,  we have:
\begin{equation}
  \label{eq:component-squares-general}
  x_k(\bn) \cdot x_k(\bn+\tfrac{1}{2}\bsigma) 
  = \frac{\prod_{i=1}^N \left( n_i + \alpha_k + \frac{k-i}{2} - \frac{1}{4}(1-\sigma_i)\right)}
  {\prod_{i \neq k} \left( \alpha_k - \alpha_i + \frac{k-i}{2} \right)},
\end{equation}
and therefore
\begin{equation}
  \label{eq:discrete-confocal-quadric-equation-general}
  \sum_{k=1}^N \frac{ x_k(\bn) x_k(\bn+\tfrac{1}{2}\bsigma) }{ n_i + \alpha_k + \frac{k-i}{2} - \frac{1}{4} (1-\sigma_i)} = 1, \qquad i=1,\ldots,N.
\end{equation}
\end{proposition}
\begin{proof} Equation \eqref{eq:component-squares-general} is obtained by straightforward computation.
  Using this result, equation \eqref{eq:discrete-confocal-quadric-equation-general}  follows from the continuous equations 
  \eqref{eq:confocal-quadrics}, \eqref{eq:elliptic-coordinates-squares}
  upon replacing $a_k = \alpha_k + \frac{k}{2}$ and $u_i = n_i - \frac{i}{2} - \frac{1}{4}(1-\sigma_i)$.
\end{proof}

We notice that \eqref{eq:component-squares-general} can be seen as a discrete version of the parametrization formulas \eqref{eq:elliptic-coordinates-squares}, while \eqref{eq:discrete-confocal-quadric-equation-general} can be seen as a discrete version of the quadric equation \eqref{eq:confocal-quadrics}. The above formulas  take the simplest form for $\bsigma=\bef=(1,\ldots,1)$:

\begin{equation}
  \label{eq:component-squares}
  x_k(\bn) \cdot x_k(\bn+\tfrac{1}{2}\bef) 
  = \frac{\prod_{i=1}^N \left( n_i + \alpha_k + \frac{k-i}{2} \right)}
  {\prod_{i \neq k} \left( \alpha_k - \alpha_i + \frac{k-i}{2} \right)}
\end{equation}
and
\begin{equation}
  \label{eq:discrete-confocal-quadric-equation}
  \sum_{k=1}^N \frac{ x_k(\bn) x_k(\bn+\tfrac{1}{2}\bef) }{ n_i + \alpha_k + \frac{k-i}{2} } = 1, \qquad i=1,...,N.
\end{equation}

In the continuous setting one can obtain from \eqref{eq:elliptic-coordinates-squares}
\begin{equation}
  \label{eq:product1}
  \scalarprod{\q(\bu)}{\q(\bu)} = \sum_{k=1}^N x_k^2(\bu) = \sum_{k=1}^N (u_k + a_k),
\end{equation}
so that the hypersurfaces $\sum_{k=1}^N u_k = {\rm const}$ are (parts) of spheres.
In particular, this implies that
\begin{equation}  \label{eq:product2}
  \left\langle \q,\frac{\partial \q}{\partial u_i}\right\rangle = \frac{1}{2},
\end{equation}
for all $i=1,\ldots,N$,
which can be regarded as a characterization of the particular parametrization 
of the confocal quadrics considered in this paper.
In the discrete case one obtains the following:

\begin{proposition}
  For any $N$-tuple of signs $\bsigma=(\sigma_1,\ldots,\sigma_N)$, $\sigma_k=\pm 1$,  we have:
  \begin{equation}  \label{eq:discrete-product1}
      \scalarprod{\q(\bn)}{\q(\bn + \tfrac{1}{2}\bsigma)} 
      = \sum_{k=1}^N \left( n_k + \alpha_k - \tfrac{1}{4}(1-\sigma_k) \right),
  \end{equation}
  and therefore, for any $i=1,\ldots,N$ and for any $\bsigma$ with $\sigma_i = -1$:
  \begin{equation}
    \label{eq:discrete-product2}
    \scalarprod{\q(\bn)}{\Delta_i\q(\bn+ \tfrac{1}{2}\bsigma)} = \frac{1}{2}.
  \end{equation}
\end{proposition}

\begin{proof} The right-hand sides of \eqref{eq:elliptic-coordinates-squares} and \eqref{eq:component-squares-general}
 may be identified by setting $a_k = \alpha_k + \frac{k}{2}$ and $u_i = n_i - \frac{i}{2} - \frac{1}{4}(1-\sigma_i)$ and hence the right-hand side of
\eqref{eq:product1} also applies in the discrete case, leading upon the above identification to \eqref{eq:discrete-product1}.
\end{proof}

Finally, we obtain a factorization similar to \eqref{eq:is prop}, \eqref{eq: s}, \eqref{eq: alphas},
which characterizes isothermic surfaces in the continuous case.
\begin{proposition} \label{prop: discrete factorization}
For $i\neq j$ we have
\begin{equation}
\scalarprod{ \Delta_i \q(\n) }{ \Delta_i \q(\n + \tfrac{1}{2}\bef)}=s^2\phi_i(n_i),
\end{equation}
\begin{equation}
\scalarprod{ \Delta_j \q(\n  - \be_j + \tfrac{1}{2}\bef) }{ \Delta_j \q(\n  - \be_j + \bef) } =s^2\phi_j(n_j),
\end{equation}
where
\begin{equation}
s(n_i,n_j)=\left | n_i-n_j+\frac{j-i}{2}+\frac{1}{2} \right |^{1/2},
\end{equation}
and
\begin{equation}
\frac{\phi_i(n_i)}{\phi_j(n_j)} = -\frac{\prod_{m\neq i,j} (n_i - n_m + \frac{m-i}{2} + \frac{1}{2})}{\prod_{m=1}^N (n_i + \alpha_m - \frac{m-i}{2} + \frac{1}{2})}
\cdot  \frac{\prod_{m=1}^N (n_j + \alpha_m - \frac{m-j}{2})}{\prod_{m\neq i,j} (n_j - n_m + \frac{m-j}{2} - \frac{1}{2})}.
\end{equation}
\end{proposition}

\begin{proof}
  For any $i, k$ we compute
  \begin{equation*}
    \Delta_i x_k(\n) \cdot \Delta_i x_k(\n + \tfrac{1}{2}\bef)
    = \frac{1}{4}\frac{1}{\prod_{m\neq k}(\alpha_k-\alpha_m+\frac{k-m}{2})} \frac{\prod_{m\neq i}(n_m+ \alpha_k + \frac{k-m}{2})}{n_i + \alpha_k + \frac{k-i}{2} + \frac{1}{2}}.
  \end{equation*}
  Here, the right-hand side can be identified with the right-hand side of the corresponding continuous expression, put as
  $$
  \left(\frac{\partial x_k}{\partial u_i} \right )^2 =\frac{1}{4} \frac{1}{\prod_{m \neq k} (a_k - a_m)}\frac{\prod_{m\neq i} (u_m+a_k)}{(u_i+a_k)},
  $$
  upon replacing 
  $$
  a_k = \alpha_k + \frac{k}{2}, \quad u_m =
  \begin{cases}
    n_m - \frac{m}{2},  &\text{for}~m\neq i,\\ 
    n_i - \frac{i}{2} + \frac{1}{2}, &\text{for}~m = i.
  \end{cases}
 $$
Therefore, the continuous identity  \eqref{eq: id squares} upon the above identification  implies
  \begin{equation*}
      \scalarprod{\Delta_i \q(\n)}{\Delta_i \q(\n + \tfrac{1}{2}\bef)}
      =  \frac{\prod_{m\neq i}(n_i - n_m + \frac{m-i}{2} + \frac{1}{2})}{4 \prod_{m=1}^N(n_i + \alpha_m - \frac{m-i}{2} + \frac{1}{2})}.  
  \end{equation*}
 Similarly, we find:
  \begin{equation*}
    \scalarprod{ \Delta_j \q(\n  - \be_j + \tfrac{1}{2}\bef) }{ \Delta_j \q(\n  - \be_j + \bef) }
    = \frac{\prod_{m\neq j}(n_j - n_m + \frac{m-j}{2} - \frac{1}{2})}{4 \prod_{m=1}^N(n_j + \alpha_m - \frac{m-j}{2})}.
  \end{equation*}
  The observation that the only factors in each of the latter two expressions which depends on both $n_i$ and $n_j$ are equal (up to sign) finishes the proof.
\end{proof}

\begin{remark}
  Similar to equations \eqref{eq:component-squares}, \eqref{eq:discrete-confocal-quadric-equation}
  it is possible to generalize Proposition \ref{prop: discrete factorization} 
  by replacing the vector $\bef$ by a  vector of signs $\bsigma$ with $\sigma_i=\sigma_j=1$ and all other components  components being arbitrary.
\end{remark}

\section{The case $N=2$}
\label{sect: 2d}

Here, and in the following section, we examine in greater detail the general theory set down in the preceding in the cases $N=2$ and $N=3$. For the benefit of the reader, these two sections are made as self-contained as possible.

\subsection{Continuous confocal coordinates}
Let $a > b > 0$. Then formulas
\begin{equation}
  \label{eq:confocal2d-map}
  \begin{aligned}
    x(u_1, u_2) &= \frac{\sqrt{u_1 + a}\sqrt{u_2 + a}}{\sqrt{a-b}},\\
    y(u_1, u_2) &= \frac{\sqrt{-(u_1 + b)}\sqrt{u_2 + b}}{\sqrt{a-b}},
  \end{aligned}
\end{equation}
define a parametrization of the first quadrant of $\R^2$ by confocal coordinates
\begin{equation*}
  \mathcal{U} = \set{ (u_1,u_2) \in \R^2 }{ -a < u_1 < -b < u_2 } \rightarrow \R_+^2.
\end{equation*}
A family of confocal conics is obtained by reflections in the coordinate axes.

\begin{figure}[H]
  \centering
  \includegraphics[width=0.49\textwidth]{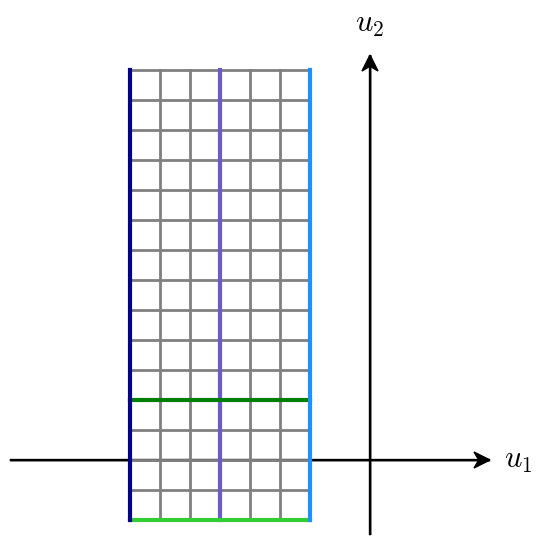}
  \includegraphics[width=0.49\textwidth]{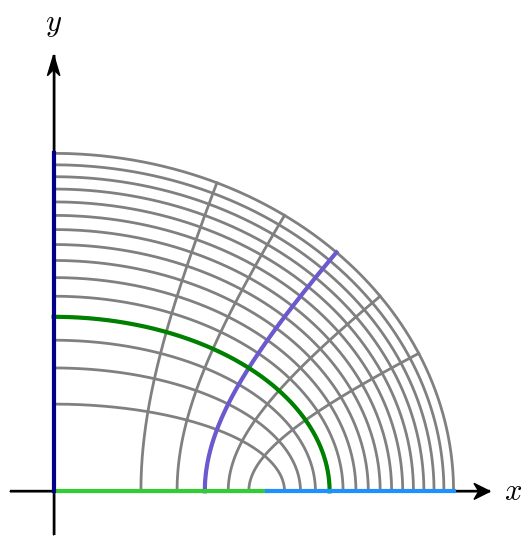}
  \caption{
    Square grid on the domain $\mathcal{U}$ and its image under the map \refeq{eq:confocal2d-map}.
    The horizontal lines $u_2=const$ are mapped to ellipses
    with the degenerate case $u_2 \searrow -b$, which is mapped to a line segment on the $x$-axis.
    The vertical lines $u_1 = const$ are mapped to hyperbolas
    with the degenerate cases $u_1 \nearrow -b$, which is mapped to a ray on the $x$-axis,
    and $u_2 \searrow -a$, which is mapped to the positive $y$-axis.
  }
\end{figure}

\subsection{Discrete confocal coordinates}
We start with the general formula 
\begin{equation*}
  \begin{aligned}
    x(n_1,n_2) &= D_1\sqr{n_1+\epsilon_1+c_1} \sqr{n_2+\epsilon_2+c_1} ,\as
    y(n_1,n_2) &=\textstyle D_2\sqr{-n_1-\epsilon_1-c_2+\frac{1}{2}}\sqr{n_2+\epsilon_2+c_2}
  \end{aligned}
\end{equation*}
for a separable solution of the discrete Euler-Poisson-Darboux system \refeq{eq:discrete-Euler-Darboux} with $\gamma=\frac{1}{2}$,
where a suitable choice of solutions \refeq{eq:discrete-separable-solution-version1}, 
\refeq{eq:discrete-separable-solution-version2} has already been made according to the continuous case. 
We use the above ansatz to illustrate the choice of the coordinate shifts $\epsilon_i$ and $c_k$ according to boundary conditions \refeq{eq:discrete-boundary-conditions1} and \refeq{eq:discrete-boundary-conditions2}.
For $\alpha, \beta \in \Z$ with $\alpha > \beta$, we define $c_1, c_2$ and $\epsilon_1, \epsilon_2$ such that
we obtain a map
\begin{equation*}
  \mathcal{U} = \set{ (n_1,n_2) \in \Z^2 }{ -\alpha \leq n_1 \leq -\beta \leq n_2 } \rightarrow \R_+^2\ ,
\end{equation*}
where the boundary components $n_1 = -\alpha$, $n_1 = -\beta$, and $n_2 = -\beta$
correspond to degenerate conics that lie on the coordinate axes:
\begin{align*}
  x |_{n_1 = -\alpha} &= 0 \quad \text{(degenerate hyperbola)},\\
  y |_{n_1 = -\beta} &= 0   \quad \text{(degenerate hyperbola)},\\
  y |_{n_2 = -\beta} &= 0   \quad \text{(degenerate ellipse)}.
\end{align*}
For this, the following linear system of equations has to be satisfied:
\begin{align*}
  \epsilon_1 + c_1 &= \alpha,\\
  \epsilon_1 + c_2 &= \beta + \frac{1}{2},\\
  \epsilon_2 + c_2&= \beta.
\end{align*}
As a consequence, we find:
\begin{equation*}
  \epsilon_2 + c_1= \alpha - \frac{1}{2}.
\end{equation*}
Thus, we end up with the formula
\bela{E5}
 \q(\bn) = \left(\bear{c} D_1\sqr{n_1 + \alpha}\sqr{n_2 + \alpha - \frac{1}{2}}\as
                               D_2\sqr{- n_1 - \beta}\sqr{n_2 + \beta}\ear\right).
\ela
Up to scaling along the coordinate axes, the latter defines discrete confocal coordinates on the first quadrant of $\R^2$, if the domain $\mathcal U$ is extended to $\mathcal U\cup\, \mathcal{U}^*$ as  demonstrated below.
From this we generate a family of discrete confocal conics by reflections in the coordinate axes.

\begin{figure}[H]
  \centering
  \includegraphics[width=0.49\textwidth]{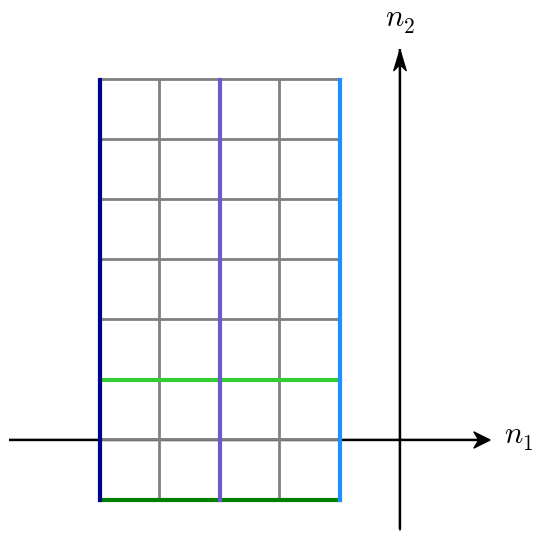}
  \includegraphics[width=0.49\textwidth]{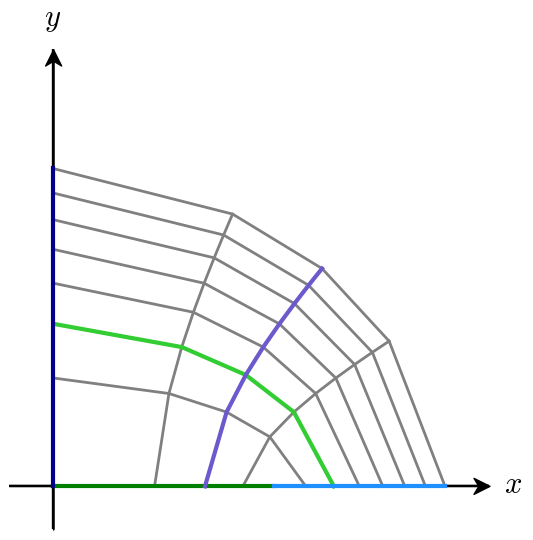}
  \caption{
    Points of the square grid on the domain $\mathcal{U}$ and their images under the map \refeq{E5}, 
    joined by straight line segments respectively.
    The horizontal lines $n_2={\rm const}$ are mapped to discrete ellipses
    with the degenerate case $n_2 = -\beta$, which is mapped to a line segment on the $x$-axis.
    The vertical lines $n_1 = {\rm const}$ are mapped to discrete hyperbolas
    with the degenerate cases $n_1 = -\beta$, which is mapped to a ray on the $x$-axis,
    and $n_1 = -\alpha$, which is mapped to the positive $y$-axis.
  }
\end{figure}

In order to implement the orthogonality condition, we extend $\q$ to $\mathcal U^*$, and compute the discrete derivatives of the extension of $\q$ along the dual edges of the two dual square lattices $\mathcal{U}$ and $\mathcal{U}^*$. Formulas \eqref{E3} for the ``discrete derivatives'' of the discrete square root immediately lead to
\bela{E7}
  \Delta_1\q(\bn) = \frac{1}{2}\left(\bear{c} 
                              \dis D_1\frac{\sqr{n_2 + \alpha - \frac{1}{2}}}{\sqr{n_1 + \alpha + \frac{1}{2}}}\AAS
                               \dis-D_2\frac{\sqr{n_2 + \beta}}{\sqr{- n_1 - \beta - \frac{1}{2}}}\ear\right)
\ela
and
\bela{E8}
  \Delta_2\q(\bn) = \frac{1}{2}\left(\bear{c} \dis D_1\frac{\sqr{n_1 + \alpha}}{\sqr{n_2 + \alpha}}\AAS
                               \dis D_2\frac{\sqr{ - n_1 - \beta}}{\sqr{n_2 + \beta + \frac{1}{2}}}\ear\right).
\ela
If we introduce the notation
\bela{E8a}
  \textstyle
  \hn{\sigma_1,\sigma_2} = \bn + \frac{1}{2}(\sigma_1,\sigma_2), \quad \sigma_i=\pm 1,
\ela
then it turns out that 
\bela{E9}
  \textstyle\left\langle \Delta_1\q(\bn), \Delta_2\q(\hn{+-}) \right\rangle= \frac{1}{4}(D_1^2 - D_2^2),
\ela
so that dual edges are orthogonal if and only if
\bela{E10}
  D_1^2 = D_2^2.
\ela
We make the choice
\bela{E23b}
  D_1^2 = D_2^2 = \frac{1}{\alpha - \beta - \frac{1}{2}}.
\ela
Formulas \eqref{E5} with the constants \eqref{E23b} constitute a discretization of the parametrization \eqref{eq:confocal2d-map}. 

\begin{figure}[H]
  \centering
  \includegraphics[width=0.49\textwidth]{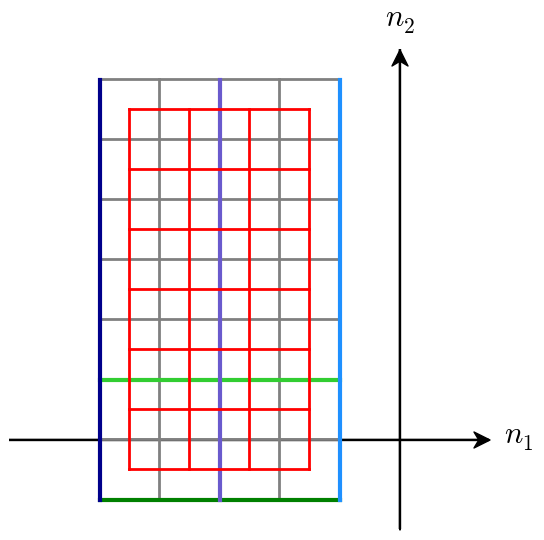}
  \includegraphics[width=0.49\textwidth]{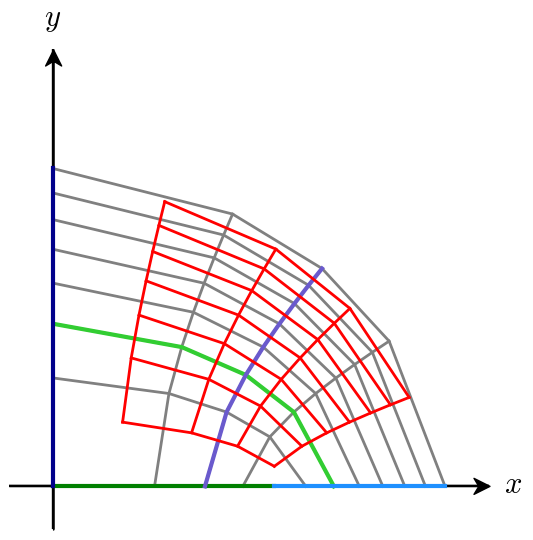}
  \caption{Points of the square grid on the domain $\mathcal{U} \cup \mathcal{U}^*$ and their images under the map \refeq{E5}.
    All pairs of corresponding dual edges are mutually orthogonal.}
\end{figure}

It is readily verified that with the choice \eqref{E23b}, a lattice point $\q(\bn)$ and its nearest neighbours $\q(\hn{++})$ and $\q(\hn{+-})$ are related by
\bela{E24b}
  \begin{aligned}
   x(\n)x(\hn{++}) & = \dis\frac{(n_1 + \alpha)(n_2 + \alpha - \frac{1}{2})}{\alpha-\beta-\frac{1}{2}},\as
   y(\n)y(\hn{++}) & = \dis\frac{(n_1 + \beta + \frac{1}{2})(n_2 + \beta)}{\beta-\alpha+\frac{1}{2}},  
  \end{aligned}
\ela
respectively by
\bela{E24c}
  \begin{aligned}
   x(\n)x(\hn{+-}) & = \dis\frac{(n_1 + \alpha)(n_2 + \alpha - 1)}{\alpha-\beta-\frac{1}{2}},\as
   y(\n)y(\hn{+-}) & = \dis\frac{(n_1 + \beta + \frac{1}{2})(n_2 + \beta -\frac{1}{2})}{\beta-\alpha+\frac{1}{2}},
  \end{aligned}
\ela
which are natural discretizations of the formulas 
\bela{E24cc}
  x^2 = \frac{(u_1 + a)(u_2 + a)}{a-b},\quad y^2 = \frac{(u_1 + b)(u_2 + b)}{b-a}
\ela
for the squares of coordinates. From \eqref{E24b}, \eqref{E24c} one easily derives
\bela{E23}
  \begin{aligned}  
   \frac{x(\bn)x(\hn{++})}{n_1 + \alpha} + \frac{y(\bn)y(\hn{++})}{n_1 + \beta + \frac{1}{2}} &= 1, \as
   \frac{x(\bn)x(\hn{++})}{n_2 + \alpha - \frac{1}{2}} + \frac{y(\bn)y(\hn{++})}{n_2 + \beta} &= 1,
  \end{aligned}
\ela
and
\bela{E23a}
  \begin{aligned}
   \frac{x(\n)x(\hn{+-})}{n_1 + \alpha} + \frac{y(\bn)y(\hn{+-})}{n_1 + \beta + \frac{1}{2}} & = 1,\as
   \frac{x(\n)x(\hn{+-})}{n_2 + \alpha - 1} + \frac{y(\bn)y(\hn{+-})}{n_2 + \beta - \frac{1}{2}} & = 1,
  \end{aligned}
\ela
which can be considered as discretizations of the defining  equations of the two confocal conics through the point $(x,y)\in\mathbb R^2$:
\bela{E23c}
  \begin{aligned}
   \frac{x^2}{u_1+a} + \frac{y^2}{u_1+b} & = 1, \as
   \frac{x^2}{u_2+a} + \frac{y^2}{u_2+b} & = 1.
  \end{aligned}
\ela

Observe that relations (\ref{E24b}) and (\ref{E24c}) may be regarded as two maps 
\bela{E24a}
  \tau^{\mbox{\tiny $++$}}:\q(\bn) \mapsto \q(\hn{++}),\quad\tau^{\mbox{\tiny $+-$}}: \q(n)\mapsto\q(\hn{+-}),
\ela
whose commutativity $\tau^{\mbox{\tiny $++$}}\circ\tau^{\mbox{\tiny $+-$}}= \tau^{\mbox{\tiny $+-$}}\circ\tau^{\mbox{\tiny $++$}}$ is directly verified. Thus, the net $\q$ can be uniquely determined from its value at a single vertex. 

Proposition \ref{prop: discrete factorization} in the case $N=2$ can be shown by simple calculations starting either with
the explicit parametrization (\ref{E5}) or the maps (\ref{E24b}), \eqref{E24c}.
For instance, a factorization property associated with  $\tau^{++}$ (shift by $\frac{1}{2}\bef$) reads:
\bela{E24ccc}
  \frac{\left\langle\Delta_1\q(\bn), \Delta_1\q(\bn+\frac{1}{2}\bef)\right\rangle}
  {\left\langle\Delta_2\q(\bn-\be_2+\frac{1}{2}\bef),\Delta_2\q(\bn-\be_2+\bef)\right\rangle} = \frac{\phi_1(n_1)}{\phi_2(n_2)},
\ela
where
\bela{E24ccca}
\frac{\phi_1(n_1)}{\phi_2(n_2)} = -\frac{(n_2+\alpha-\frac{1}{2})(n_2+\beta)}{(n_1+\alpha+\frac{1}{2})(n_1+\beta+1)},
\ela
and a similar property is associated with the map $\tau^{\mbox{\tiny $+-$}}$. This can be seen as a discretization of the isothermicity property of the system of confocal conics which reads
\bela{E24cccc}
  \frac{|\partial \q/\partial u_1|^2}{|\partial \q/\partial u_2|^2} = \frac{\alpha_1(u_1)}{\alpha_2(u_2)},
\ela
where
\bela{24ccccc}
 \frac{\alpha_1(u_1)}{\alpha_2(u_2)} = -\frac{(u_2+a)(u_2+b)}{(u_1+a)(u_1+b)}.
\ela
The combinatorics of the factorization property \eqref{E24ccc} is illustrated in \linebreak Figure \ref{fig:isothermic-2d}.

\begin{figure}[H]
  \centering
  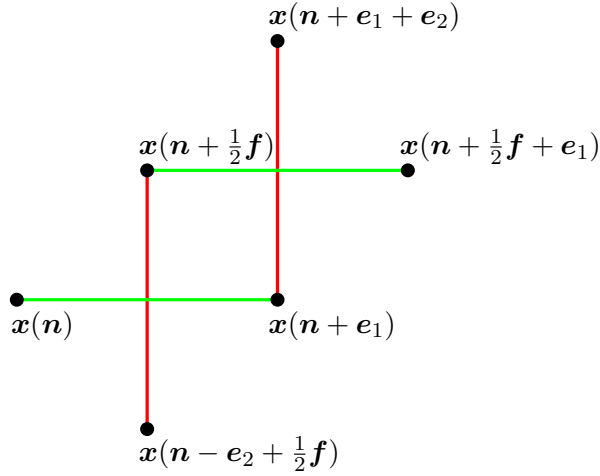
  \caption{
    Combinatorics of the factorization property \eqref{E24ccc}.}
  \label{fig:isothermic-2d}
\end{figure}

\section{The case $N=3$}
\label{sect: 3d}

\subsection{Continuous confocal coordinates}
Let $a > b > c > 0$. Then formulas
\begin{align}\label{eq: param cont 3d}
  x(u_1, u_2, u_3) &= \frac{\sqrt{u_1+a}\sqrt{u_2+a}\sqrt{u_3+a}}{\sqrt{a-b}\sqrt{a-c}},\nonumber\\
  y(u_1, u_2, u_3) &= \frac{\sqrt{-(u_1+b)}\sqrt{u_2+b}\sqrt{u_3+b}}{\sqrt{a-b}\sqrt{b-c}},\\
  z(u_1, u_2, u_3) &= \frac{\sqrt{-(u_1+c)}\sqrt{-(u_2+c)}\sqrt{u_3+c}}{\sqrt{a-c}\sqrt{b-c}}\nonumber
\end{align}
define a parametrization of the first octant of $\R^3$ by confocal coordinates,
\begin{equation*}
  \mathcal{U} = \set{ (u_1,u_2,u_3) }{ -a < u_1 < -b < u_2 < -c < u_3 } \rightarrow \R_+^3\ .
\end{equation*}
Confocal quadrics are obtained by reflections of the coordinate surfaces (corresponding to $u_i={\rm const}$ for $i=1,2$ or 3) in the coordinate planes of $\R^3$, see Figure \ref{fig:confocal3d}, left.
\begin{itemize}
\item The planes $u_3 = {\rm const}$ are mapped to ellipsoids.
  In the degenerate case $u_3 \searrow -c$ one has $z = 0$,
  while $x(u_1,u_2)$ and $y(u_1,u_2)$ exactly recover the two-dimensional case \refeq{eq:confocal2d-map}
  on the interior of an ellipse given by $u_2 \nearrow -c$.
\item The planes $u_2 = {\rm const}$ are mapped to one-sheeted hyperboloids
  with the two degenerate cases corresponding to $u_2 \nearrow -c$ and $u_2 \searrow -b$.
\item The planes $u_1 = {\rm const}$ are mapped to two-sheeted hyperboloids
  with the two degenerate cases corresponding to $u_1 \nearrow -b$ and $u_1 \searrow -a$.
\end{itemize}

\subsection{Discrete confocal quadrics}
Let $\alpha, \beta, \gamma \in \Z$ with $\alpha > \beta > \gamma$. Then the formula
\bela{E15}
 \q(\bn) = \left(\bear{c} D_1\sqr{n_1 + \alpha}\sqr{n_2 + \alpha - \frac{1}{2}}\sqr{n_3 + \alpha - 1}\as
                               D_2\sqr{- n_1 - \beta}\sqr{n_2 + \beta}\sqr{n_3 + \beta - \frac{1}{2}}\as
                               D_3\sqr{- n_1 - \gamma - \frac{1}{2}}\sqr{- n_2 - \gamma}\sqr{n_3 + \gamma}\ear\right)
\ela
with $\q(\n) = (x(\n),y(\n),z(\n))$ defines a discrete net in the first octant of $\R^3$ (discrete confocal coordinate system), that is, a map
\begin{equation*}
  \mathcal{U} = \set{ (n_1,n_2,n_3) \in \Z^3 }{ -\alpha \leq n_1 \leq -\beta \leq n_2 \leq -\gamma \leq n_3 } \rightarrow \R_+^3
\end{equation*}
which is a separable solution of (dEPD$_{1/2}$). If this net is extended to $\mathcal U \cup\, \mathcal U^*$ then discrete confocal quadrics are obtained by reflections of the coordinate surfaces ($n_i={\rm const}$ for $i=1,2$ or 3) in the coordinate planes of $\R^3$, see Figure \ref{fig:confocal3d}, right, provided that the constants $D_k$ are chosen in the manner described below.
The five boundary components $n_1 = -\alpha$, $n_1 = -\beta$, $n_2 = -\beta$, $n_2= -\gamma$, and $n_3 = -\gamma$
are mapped to degenerate quadrics that lie in the coordinate planes of $\R^3$.

One computes the discrete derivatives with the help of formulas \eqref{E3}:
\bela{E17}
  \Delta_1\q(\bn) = \frac{1}{2}\left(\bear{c} 
             \dis D_1\frac{\sqr{n_2 + \alpha - \frac{1}{2}}\sqr{n_3 + \alpha - 1}}{\sqr{n_1 + \alpha + \frac{1}{2}}}\AAS
             \dis-D_2\frac{\sqr{n_2 + \beta}\sqr{n_3 + \beta - \frac{1}{2}}}{\sqr{- n_1 - \beta - \frac{1}{2}}}\AAS
	        \dis -D_3\frac{\sqr{- n_2 - \gamma}\sqr{n_3 + \gamma}}{\sqr{- n_1 - \gamma - 1}}\ear\right),
\ela
\bela{E18}
  \Delta_2\q(\bn) = \frac{1}{2}\left(\bear{c} 
           \dis D_1\frac{\sqr{n_1 + \alpha}\sqr{n_3 + \alpha - 1}}{\sqr{n_2 + \alpha}}\AAS
           \dis D_2\frac{\sqr{ - n_1 - \beta}\sqr{n_3 + \beta - \frac{1}{2}}}{\sqr{n_2 + \beta + \frac{1}{2}}}\AAS
           \dis -D_3\frac{\sqr{- n_1 - \gamma - \frac{1}{2}}\sqr{n_3 + \gamma}}{\sqr{- n_2 - \gamma - \frac{1}{2}}}
\ear\right) 
\ela
and
\bela{E18a}
  \Delta_3\q(\bn) = \frac{1}{2}\left(\bear{c} 
           \dis D_1\frac{\sqr{n_1 + \alpha}\sqr{n_2 + \alpha - \frac{1}{2}}}{\sqr{n_3 + \alpha - \frac{1}{2}}}\AAS
           \dis D_2\frac{\sqr{- n_1 - \beta}\sqr{n_2 + \beta}}{\sqr{n_3 + \beta}}\AAS
           \dis D_3\frac{\sqr{- n_1 - \gamma - \frac{1}{2}}\sqr{- n_2 - \gamma}}{\sqr{n_3 + \gamma + \frac{1}{2}}}\ear\right).
\ela
In accordance with the general orthogonality condition, we now demand that dual pairs of edges and faces of the nets $\q(\mathcal{U})$ and $\q(\mathcal{U}^*)$ be orthogonal, so that
\bela{E19}
  \begin{aligned}
   \langle\Delta_1\q(\bn),\Delta_2\q(\bn-\be_2+\tfrac{1}{2}\bef)\rangle & = 0,\\
   \langle\Delta_1\q(\bn),\Delta_3\q(\bn-\be_3+\tfrac{1}{2}\bef)\rangle & = 0,\\
   \langle\Delta_2\q(\bn),\Delta_3\q(\bn-\be_3+\tfrac{1}{2}\bef)\rangle & = 0.
  \end{aligned}
\ela
Evaluation of the above conditions leads to
\bela{E20}
  \begin{aligned}
   \textstyle D_1^2(n_3 +  a - \frac{3}{2}) - D_2^2(n_3 + b - \frac{3}{2}) + D_3^2(n_3 + c - \frac{3}{2}) & = 0,\\
   \textstyle D_1^2(n_2 + a - 1) - D_2^2(n_2 + b - 1) + D_3^2(n_2 + c - 1) & = 0,\\
   \textstyle D_1^2(n_1 + a - \frac{1}{2}) - D_2^2(n_1 + b - \frac{1}{2}) + D_3^2(n_1 + c - \frac{1}{2}) & = 0,
  \end{aligned}
\ela
where 
\bela{Z7}
 \textstyle
  a = \alpha + \frac{1}{2},\quad b = \beta+1,\quad c = \gamma+\frac{3}{2}.
\ela
These are, {\sl mutatis mutandis}, identical with their classical continuous counterparts as demonstrated in connection with the general case analyzed in Section \ref{sect: discrete quadrics}. Since the coefficients $D_i$ are independent of, for instance, $n_3$, the first condition in (\ref{E20}) splits into the pair
\bela{E21}
  \begin{aligned}
   D_1^2 - D_2^2 + D_3^2 & =  0,\\ 
   D_1^2 a - D_2^2 b + D_3^2 c & = 0,
  \end{aligned}
\ela
and it is evident that the remaining two conditions constitute linear combinations thereof. Accordingly, the orthogonality requirement leads to the unique relative scaling
\bela{E22}
  \frac{D_1^2}{b-c} = \frac{D_2^2}{a-c} = \frac{D_3^2}{a-b} =: D^2.
\ela

\begin{figure}[H]
  \centering
  \includegraphics[clip, trim={380 200 380 200}, width=0.8\textwidth]{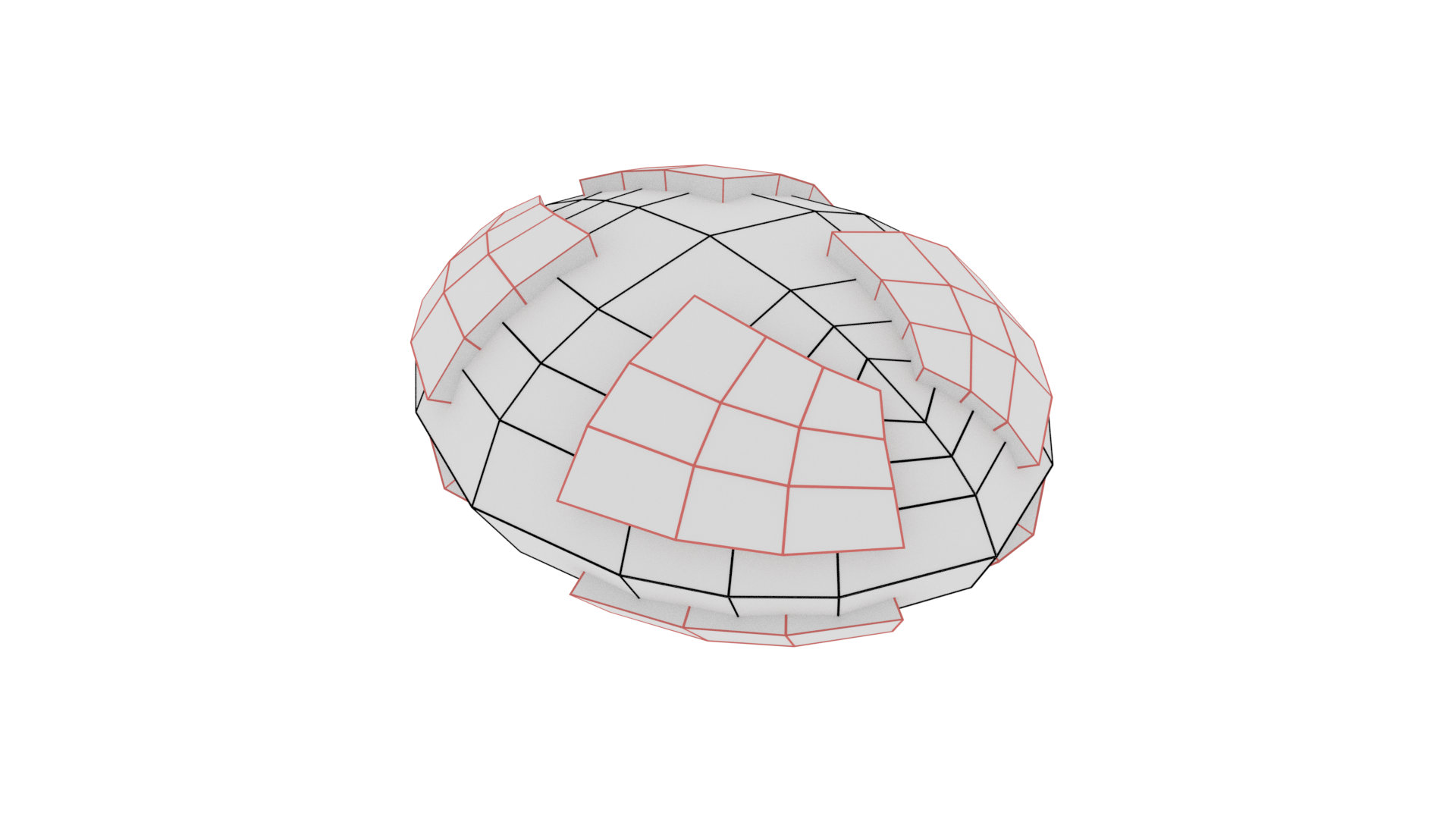}
    \caption{
    A discrete ellipsoid ($n_1, n_2$ integer, $n_3 = {\rm const}$) from the system \refeq{E15}
    and its two adjacent layers from the dual n ($n_1, n_2$ half-integer, $n_3 = {\rm const} \pm \frac{1}{2}$).
    All faces are planar and orthogonal to the corresponding edges of the other net.}
\end{figure}

\begin{remark}
The curvature lines of a smooth surface are characterized by the following properties: they form a conjugate net, and along each curvature line two infinitesimally close normals intersect. In the case of discrete confocal coordinates  the edges of the dual net $\q(\mathcal U^*)$ can be interpreted as normals to the faces of the net $\q(\mathcal U)$.  Since both nets have planar faces, any two neighboring normals intersect. Thus, extended edges of $\q(\mathcal U^*)$ constitute a \emph{discrete line congruence} normal to the faces of the Q-net (discrete conjugate net) $\q(\mathcal U)$ (cf. \cite{bobenko-suris} for the notion of a discrete line congruence).
  \end{remark}

The bilinear relations between a lattice point $\q(\bn)$ and its nearest neighbours $\q(\bn+\frac{1}{2}\bsigma)$ may be formulated as follows:
\bela{E23d}
  \begin{aligned}
\dis\frac{x(\bn)x(\bn+\frac{1}{2}\bsigma)}{u+a} + \frac{y(\bn)y(\bn+\frac{1}{2}\bsigma)}{u+b} 
   + \frac{z(\bn)z(\bn+\frac{1}{2}\bsigma)}{u+c} & = 1,\as
\dis\frac{x(\bn)x(\bn+\frac{1}{2}\bsigma)}{v+a} + \frac{y(\bn)y(\bn+\frac{1}{2}\bsigma)}{v+b} 
   + \frac{z(\bn)z(\bn+\frac{1}{2}\bsigma)}{v+c} & = 1, \as
\dis\frac{x(\bn)x(\bn+\frac{1}{2}\bsigma)}{w+a} + \frac{y(\bn)y(\bn+\frac{1}{2}\bsigma)}{w+b} 
   + \frac{z(\bn)z(\bn+\frac{1}{2}\bsigma)}{w+c} & = 1,   
  \end{aligned}
\ela
provided that
\bela{E23e}
  D^2 = \frac{1}{(a-b)(a-c)(b-c)},
\ela
and
\bela{Z8}
u=n_1+\tfrac{1}{4}\sigma_1-\tfrac{3}{4}, \quad v=n_2+\tfrac{1}{4}\sigma_2-\tfrac{5}{4}, \quad w=n_3+\tfrac{1}{4}\sigma_3-\tfrac{7}{4}.
\ela

\subsection{Discrete umbilics and discrete focal conics}

An interesting feature of discrete confocal quadrics which is not present in the two-dimensional case is obtained by considering the ``discrete umbilics'' (that is, vertices of valence different from 4) of the discrete ellipsoids $n_3 =\mbox{const}$ and the discrete two-sheeted hyperboloids $n_1=\mbox{const}$. In the case of the discrete ellipsoids, these have valence 2 and are located at $n_1=n_2=-\beta$ so that (\ref{E15}) reduces to
the planar discrete curve
\bela{G1}
  \q(n_3) = \left(\bear{c} D_1(\alpha-\beta-\frac{1}{2})\sqr{n_3 + \alpha - 1}\as
                               0\\
                               D_3(\beta - \gamma - \frac{1}{2})\sqr{n_3 + \gamma}\ear\right).
\ela
Once again, it turns out convenient to extend the domain of this one-dimensional lattice to the appropriate subset of $\Z\cup\Z^*$ so that
\bela{G2}
  \frac{x(n_3)x(n_3+\frac{1}{2})}{\alpha-\beta-\frac{1}{2}} - \frac{z(n_3)z(n_3+\frac{1}{2})}{\beta-\gamma-\frac{1}{2}} = 1.
\ela
The latter constitutes a discretization of the focal hyperbola \cite{sommerville}
\bela{G3}
  \frac{x^2}{a-b} - \frac{z^2}{b-c} = 1
\ela
which is known to be the locus of the umbilical points of confocal ellipsoids. Similarly, evaluation of (\ref{E15}) at $n_2=n_3=-\gamma$ produces the planar discrete curve
\bela{G4}
   \q(n_1) = \left(\bear{c} D_1(\alpha - \gamma - 1)\sqr{n_1 + \alpha}\as
                               D_2(\beta - \gamma - \frac{1}{2})\sqr{-n_1 - \beta}\as
                               0
  \ear\right)
\ela
which consists of the discrete umbilics of the discrete two-sheeted hyperboloids. Extension to half-integers yields
\bela{G5}
  \frac{x(n_1)x(n_1+\frac{1}{2})}{\alpha-\gamma-1} + \frac{y(n_1)y(n_1+\frac{1}{2})}{\beta-\gamma-\frac{1}{2}} = 1,
\ela
which reproduces, in the formal continuum limit, the classical focal ellipse
\bela{G6}
  \frac{x^2}{a-c} + \frac{y^2}{b-c} = 1
\ela
as the locus of the umbilical points of confocal two-sheeted hyperboloids.

\appendix
\section{Incircular nets as orthogonal Koenigs nets}
\label{sect: incircular}

A geometric discretization of confocal conics as incircular nets (IC-nets) was recently suggested in \cite{Akopyan-Bobenko}.
This version of discrete confocal conics is given via a simple local geometric condition: one considers a congruence of straight lines with the combinatorics of the square grid such that all the quadrilaterals formed by neighboring lines possess inscribed circles. In this appendix we show that, surprisingly, IC-nets share two crucial properties with discrete confocal coordinates introduced in the present paper, namely the Koenigs property and the orthogonality in the sense of Definition \ref{def: d ortho}. One should mention however that IC-nets are not separable, therefore they do not constitute a special case of discrete confocal conics as defined in Defintion \ref{def:discrete-elliptic-coordinates}.

\begin{definition}
  A discrete net $f : \Z^2 \supset U \rightarrow \R^2$ is called an \emph{incircular net (IC-net)} if
  \begin{enumerate}
  \item The points $f_{ij}$ with $i={\rm const}$, respectively $j={\rm const}$, lie on straight lines,
    preserving the order.
  \item Every elementary quadrilateral $(f_{ij}, f_{i+1, j}, f_{i+1, j+1}, f_{i, j+1})$ has an incircle.
  \end{enumerate}
\end{definition}

All lines of an IC-net touch some conic $\alpha$, while all vertices of one diagonal $i + j = {\rm const}$, resp. $i - j = {\rm const}$, lie on a conic confocal to $\alpha$.

Denote the incenter of the quadrilateral $(f_{ij}, f_{i+1, j}, f_{i+1, j+1}, f_{i, j+1})$ by $\omega_{ij}$.  So, $\omega : U \rightarrow \R^2$ is the net of incenters of $f$. Note that $\omega$ also possesses property (i). Denote the two dual subnets of $\omega$, corresponding to $(i,j)$ with $2k \coloneqq i+j$ and $2l \coloneqq j-i$ even, respectively odd, by $\eta$ and $\tilde{\eta}$:
\[
  \eta_{kl} \coloneqq \omega_{k-l, k+l},  \; (k,l) \in \Z^2 \quad {\rm and}\quad 
  \tilde{\eta}_{kl} \coloneqq \omega_{k-l, k+l}, \; (k,l) \in {(\Z^2)}^*.
\]
In Figure \ref{fig:icnet-centers-with-conics}, the edges of the nets $\eta$ and $\tilde\eta$ are shown. The intersection points of dual pairs of edges happen to be points of the underlying IC-net $f$. At each such point, the intersecting edges of $\eta$ and of $\tilde\eta$ are tangent to the confocal conics mentioned above (the conics through $f_{ij}$ with $i+j=2k={\rm const}$, resp. with $j-i=2l={\rm const}$).  Therefore, the dual pairs of edges are orthogonal. We show that these nets also possess the Koenigs property and collect their important properties in the following theorem.

\begin{figure}[H]
\begin{center}
  \includegraphics[width=0.8\textwidth]{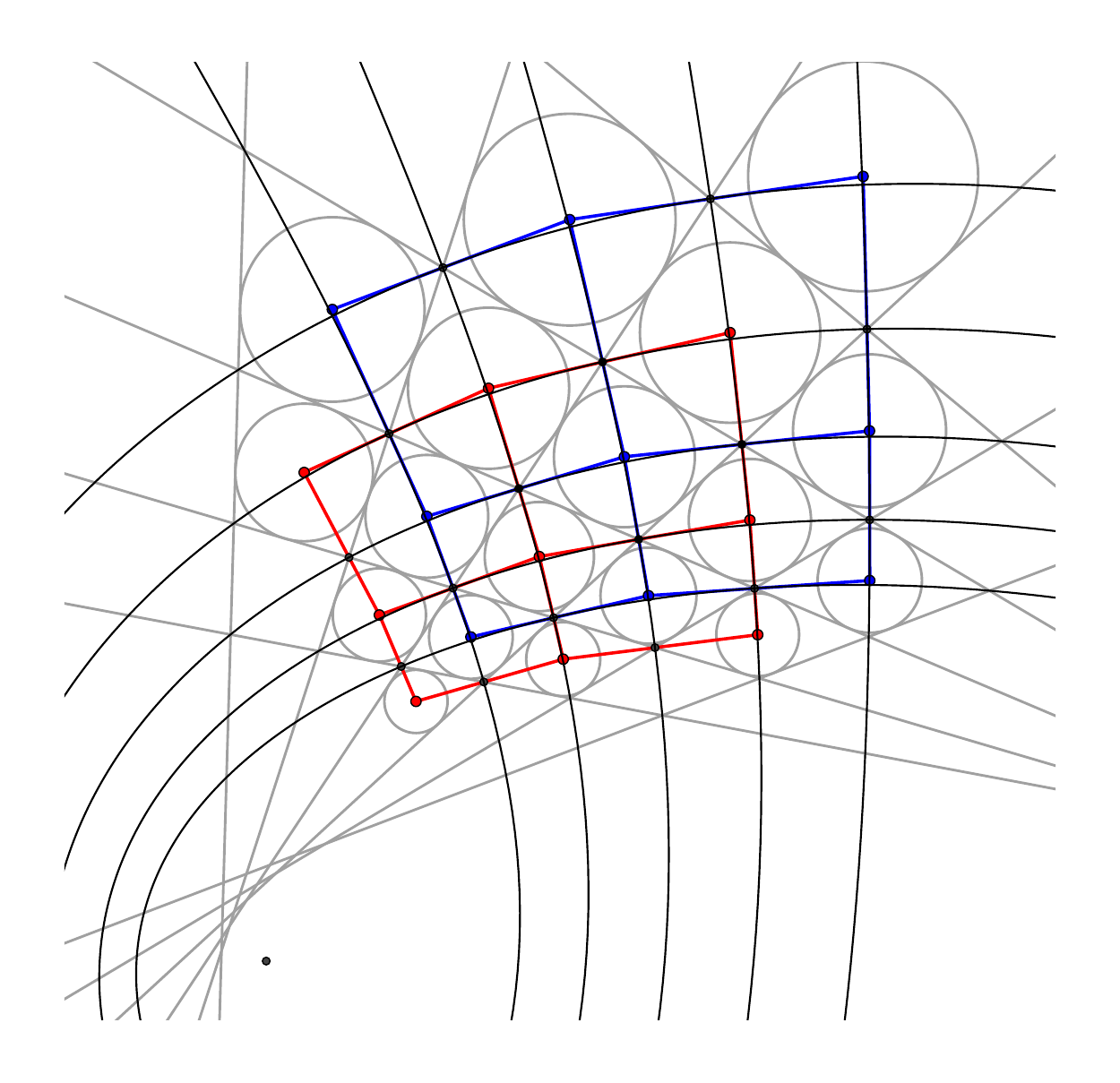}
\end{center}  
  \caption{
    Two dual subnets $\eta$ and $\tilde\eta$ of the net of incenters of an IC-net.
    The edges are tangent to confocal conics, and corresponding edges of the two nets are orthogonal.
  }
  \label{fig:icnet-centers-with-conics}
\end{figure}

\begin{theorem} \label{thm:icnet-centers-discrete-confocal-conics}
  For the two dual subnets $\eta$ and $\tilde{\eta}$ of the incenter-net of an IC-net:
\begin{enumerate}
  \item[{\rm (i)}] the edges are tangent to confocal conics, the points of tangency being the points of the IC-net;
  \item[{\rm (ii)}] each subnet consists of intersection points of diagonals of elementary quadrilaterals of the other subnet;
  \item[{\rm (iii)}] both subnets are circular-conical (that is, opposite angles sum up to $\pi$ in each quadrilateral and at each vertex-star);
  \item[{\rm (iv)}] each pair of dual edges intersects orthogonally;
  \item[{\rm (v)}] both subnets are Koenigs nets.
\end{enumerate}  
\end{theorem}

\begin{proof}\
  \begin{enumerate}
  \item See \cite{Akopyan-Bobenko}.
  \item This holds for the two dual subnets of any net consisting of straight lines.
  \item Each of the dual nets corresponds to the incenters of a checkerboard IC-net, that is, a net having incircles in every other quadrilateral
    (both checkerboard IC-nets fitting perfectly into each other forming a regular IC-net). Checkerboard IC-nets have been observed to be circular-conical in   \cite{Akopyan-Bobenko}.
  \item 
    Dual edges intersect at a point where two lines of the IC-net intersect.
    The dual edges of $\eta$ and of $\tilde\eta$ are the two angle bisectors of those two lines of the IC-net, and therefore are mutually orthogonal.
  \item From now on, we use the shift notation, like in Section \ref{sect: dKoenigs}, so that $\eta_{(\pm i)}(\bn) \coloneqq \eta(\bn \pm \be_i)$ etc., with the understanding that the argument of $\eta$, $\tilde\eta$ is $(k,l)$, while the argument of $f$ is $(i,j)$. Consider four quadrilaterals of the net $\eta$ adjacent to one vertex  (compare Figure \ref{fig:icnet-koenigs-proof}). 
    The points of intersection of diagonals of these four quadrilaterals are the points $\tilde{\eta}$, $\tilde{\eta}_{(1)}$, $\tilde{\eta}_{(12)}$, $\tilde{\eta}_{(2)}$ of the net $\tilde\eta$.
    We show that
    \begin{equation}
      \label{eq:Koenigs-Menelaus}
      \multiratio{\eta_{(1)}}{\tilde{\eta}_{(12)}}{\eta_{(2)}}{\tilde{\eta}_{(2)}}{\eta_{(-1)}}{\tilde{\eta}}{\eta_{(-2)}}{\tilde{\eta}_{(1)}} = 1.
    \end{equation}
    This is equivalent to the net $\eta$ being Koenigs (see \cite[p. 52]{bobenko-suris}).

    Considering one of the four quotients on the left-hand side,  we find:
    \begin{equation*}
      \frac{|\eta_{(1)} \tilde{\eta}_{(12)}|}{|\tilde{\eta}_{(12)} \eta_{(2)}|} 
      = \frac{\text{area}(\eta_{(1)}, \tilde{\eta}_{(12)}, \eta)}{\text{area}(\tilde{\eta}_{(12)}, \eta_{(2)}, \eta)} 
      = \frac{|f_{(1)} \tilde{\eta}_{(12)}| \cdot |\eta \eta_{(1)}|}{|f_{(12)} \tilde{\eta}_{(12)}| \cdot |\eta \eta_{(2)}|},
    \end{equation*}
    since the dual edges of $\eta$ and of $\tilde\eta$ are orthogonal.
    In the product the lengths of the edges $|\eta \eta_{(i)}|$ cancel out, and we obtain
    \begin{eqnarray*}
      \lefteqn{\multiratio{\eta_{(1)}}{\tilde{\eta}_{(12)}}{\eta_{(2)}}{\tilde{\eta}_{(2)}}{\eta_{(-1)}}{\tilde{\eta}}{\eta_{(-2)}}{\tilde{\eta}_{(1)}}}\\
      & = & \left( 
        \multiratio{\tilde{\eta}}{f}{\tilde{\eta}_{(1)}}{f_{(1)}}{\tilde{\eta}_{(12)}}{f_{(12)}}{\tilde{\eta}_{(2)}}{f_{(2)}}
        \right)^{-1}.
    \end{eqnarray*}
    The latter product is equal to 1 since the triangles $(\tilde{\eta}, f, f_{(2)})$ and $(\tilde{\eta}_{(12)}, f_{(1)}, f_{(12)})$
    are perspective triangles (Menelaus condition for Desargues configuration, cf. \cite[p. 361]{bobenko-suris}). We mention that the right-hand side of the latter formula being equal to 1 is the Koenigs condition for the net $\omega$, while the left-hand side being equal to 1 is the Koenigs condition for the net $\eta$. \qedhere
  \end{enumerate}
\end{proof}

Apparently, there also holds: 
\begin{itemize}
\item[(vi)] {\em the dual subnets $\eta$ and $\tilde\eta$ satisfy the discrete factorization property \eqref{E24ccc}},
\end{itemize}
but at present this only has been checked via numerical experiments.

\begin{figure}[H]
  \includegraphics[width=\textwidth]{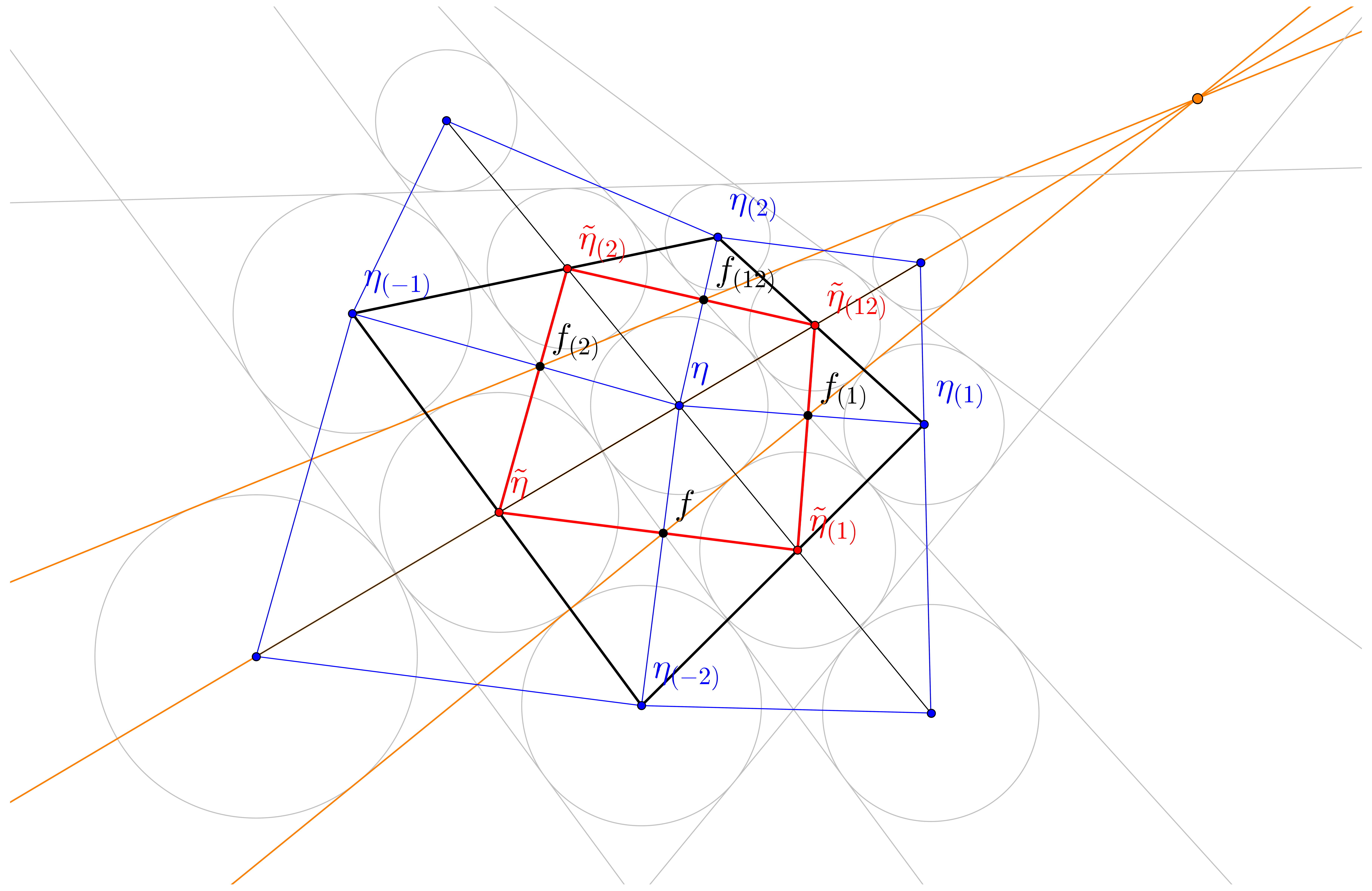}
  \caption{
    Koenigs property of the incenter net $\omega$ of an IC-net
    implies the Koenigs property for its two diagonal nets $\eta$ and $\tilde{\eta}$.
  }
  \label{fig:icnet-koenigs-proof}
\end{figure}

%
%
\enlargethispage{2\baselineskip}

\end{document}